\documentclass[11pt]{article}
\usepackage[plain]{fullpage}

\usepackage{etoolbox}
\providetoggle{long}
\usepackage[none]{hyphenat}
\settoggle{long}{true}

\usepackage[pagewise,displaymath, mathlines]{lineno}
\usepackage[ruled,vlined,linesnumbered]{algorithm2e}
\usepackage{amssymb,amsmath,amsthm}
\usepackage{graphicx,enumerate}
\usepackage{color,cite}
\usepackage{url}
\usepackage[text={17cm,24cm}]{geometry}
\usepackage{etoolbox}
\DeclareMathOperator*{\MCH}{{\mathcal{C}}}
\DeclareMathOperator*{\MNH}{{\mathcal{N}}}

\newcommand{\NP}{\ensuremath{\mathsf{NP}}}

\def\EnableMNotes{1}

\ifnum\EnableMNotes=1
\newcommand{\mynote}[1]{{\sf #1}}
\else
\newcommand{\mynote}[1]{}
\fi

\newtheorem{theorem}{Theorem}[section]

\newtheorem{claim2}{Claim}
\newtheorem{assumption}{Assumption}
\newtheorem{proposition}[theorem]{Proposition}
\newtheorem{corollary}[theorem]{Corollary}
\newtheorem{definition}[theorem]{Definition}
\newtheorem{lemma}[theorem]{Lemma}
\newtheorem*{diamondLemma}{Lemma~\ref{diamondLemma}}
\newtheorem{example}[theorem]{Example}

\newtheorem*{question}{Question}

\theoremstyle{remark}
\newtheorem{remark}[theorem]{Remark}

\numberwithin{subcase}{case}

\numberwithin{subsubc}{subcase}
\theoremstyle{remark}
\newtheorem*{cproof}{Proof}

\newcounter{caseCount}
\newenvironment{thmcase}{\medskip
    \noindent\refstepcounter{caseCount}\textit{Case~\arabic{caseCount}: }}{\medskip}
\AtBeginEnvironment{cproof}{\setcounter{caseCount}{0}}

\newcounter{subcaseCount}[caseCount]

\newcounter{subsubcCount}[subcaseCount]

\begin{document}
\title{Linear separation of connected dominating sets in graphs
\footnote{A part of this work appeared as an extended abstract in~\cite{ISAIM2014}.}}

\author{Nina Chiarelli\\
\small University of Primorska, UP FAMNIT, Glagolja\v ska 8, SI6000 Koper, Slovenia\\
\small University of Primorska, UP IAM, Muzejski trg 2, SI6000 Koper, Slovenia\\
\small \texttt{nina.chiarelli@famnit.upr.si}\\
\and
Martin Milani\v c\\
\small University of Primorska, UP IAM, Muzejski trg 2, SI6000 Koper, Slovenia\\
\small University of Primorska, UP FAMNIT, Glagolja\v ska 8, SI6000 Koper, Slovenia\\
\small \texttt{martin.milanic@upr.si}
}

\date{\today}
\pagestyle{plain}
\maketitle
\begin{abstract}
A connected dominating set in a graph is a dominating set of vertices that induces a connected subgraph.
Following analogous studies in the literature related to independent sets, dominating sets, and total dominating sets, we study in this paper the class of graphs in which the connected dominating sets can be separated from the other vertex subsets by a linear weight function. More precisely, we say that a graph is connected-domishold if it admits non-negative real weights associated to its vertices such that a set of vertices is a connected dominating set if and only if the sum of the corresponding weights exceeds a certain threshold.
We characterize the graphs in this non-hereditary class in terms of a property of the set of minimal cutsets of the graph.
We give several characterizations for the hereditary case, that is, when each connected induced subgraph is required to be connected-domishold.
The characterization by forbidden induced subgraphs implies that the class properly generalizes two well known classes of chordal graphs, the block graphs and the trivially perfect graphs. Finally, we study certain algorithmic aspects of connected-domishold graphs. Building on connections with minimal cutsets and properties of the derived hypergraphs and Boolean functions, we show that our approach leads to new polynomially solvable cases of the weighted connected dominating set problem.

\medskip

\begin{sloppypar}
\end{sloppypar}

\medskip
\noindent{\bf Math.~Subj.~Class.}~(2010): 05C69, 05C75, 05C65, 05C85
\end{abstract}

\section{Introduction}

\subsection{Background}

Threshold concepts have been a subject of investigation for various discrete structures, including graphs~\cite{MR3281177,MR0479384,MR1417258}, Boolean functions~\cite{MR0128319,MR2742439,conf/focs/Elgot60,Gabelman,MR0439441,MR798011}, and hypergraphs~\cite{MR2063679,MR791660}. A common theme of these studies is a quest for necessary and sufficient conditions for the property that a given combinatorial structure defined over some finite ground set $U$ admits non-negative real weights associated to elements of $U$ such that a subset of $U$ satisfies a certain property, say $\pi$, if and only if the sum of the corresponding weights exceeds a certain threshold.
A more general framework has also been proposed, where the requirement is that a subset of $U$ satisfies property $\pi$ if and only if the sum of the corresponding weights belongs to a set $T$ of thresholds given by a membership oracle~\cite{MR2823204}.

\smallskip
Having the set $U$ equipped with weights as above can have useful algorithmic implications. Consider for example the optimization problem of finding a subset of $U$ with property $\pi$ that has either maximum or minimum cost (according to a given linear cost function on the elements of the ground set). It was shown in~\cite{MR2823204} that if the weights as above are known and integer, then the problem can be solved by a dynamic programming approach in time $\mathcal{O}(|U|M)$ and with $M$ calls of the membership oracle, where $M$ is a given upper bound for $T$.
The pseudo-polynomial running time should be expected, since the problem is very general and captures also the well-known knapsack problem~\cite{MR2161720}. Note, however, that the problem admits a much simpler, polynomial time solution in the special case when the costs are unit and if we assume the monotone framework, where a set satisfies property $\pi$ as soon as its total weight exceeds a certain threshold. Under these assumptions, a minimum-sized subset of $U$ satisfying property $\pi$ can be found by a simple greedy algorithm starting with the empty set and adding the elements in order of non-increasing weight until the threshold is exceeded.

\smallskip
Many interesting graph classes can be defined within the above framework, including threshold graphs~\cite{MR0479384,MR1417258,MR2359603}, domishold graphs~\cite{MR0491342}, total domishold graphs~\cite{ChiMil13,MR3281177}, equistable graphs~\cite{MR553649}, and equidominating graphs~\cite{MR553649}.
In general, the properties of the resulting graph classes depend both on the choice of property $\pi$ and on the constraints imposed on the structure of the set $T$ of thresholds. For example, if $U$ is the vertex set of a graph, property $\pi$ denotes the property of being an independent (stable) set in a graph, and $T$ is restricted to be an interval unbounded from below, we obtain the class of threshold graphs, which is a very well understood class of graphs, admitting many characterizations and linear time algorithms for recognition and various optimization problems (see, e.g.,~\cite{MR1417258}). If $\pi$ denotes the property of being a dominating set and $T$ is an interval unbounded from above, we obtain the class of domishold graphs, which enjoys similar properties as the class of threshold graphs. On the other hand, if $\pi$ is the property of being a \emph{maximal} stable set and $T$ is restricted to consist of a single number, we obtain the class of equistable graphs, for which the recognition complexity is open (see, e.g.,~\cite{DBLP:conf/wg/LevitMT12}), no structural characterization is known, and several \NP-hard optimization problems remain intractable~\cite{MR2823204}.

\smallskip
Notions and results from the theory of Boolean functions~\cite{MR2742439} and hypergraphs~\cite{MR1013569} can be useful for the study of graph classes defined within the above framework. For instance, the characterization of hereditarily total domishold graphs in terms of forbidden induced subgraphs given in~\cite{MR3281177} is based on the facts that every threshold Boolean function is $2$-asummable~\cite{MR0128319} and that every dually Sperner hypergraph is threshold~\cite{ChiMil13}.\footnote{In~\cite{ChiMil13,MR3281177}, the hereditarily total domishold graphs were named hereditary total domishold graphs. We prefer to adopt the grammatically more correct term ``hereditarily total domishold''.} Moreover, the fact that threshold Boolean functions are closed under dualization and (when given by their complete DNF) can be recognized in polynomial time~\cite{MR798011} leads to efficient algorithms for recognizing total domishold graphs and for finding a minimum total dominating set in a given total domishold graph~\cite{ChiMil13}. The relationship also goes the other way around. For instance, total domishold graphs can be used to characterize threshold hypergraphs and threshold Boolean functions~\cite{MR3281177}.

\subsection{Aims and motivation}

The aim of this paper is to further explore and exploit this fruitful interplay between threshold concepts in graphs, hypergraphs, and Boolean functions. We do this by studying the class of \emph{connected-domishold} graphs, a new class of graphs that can be defined in the above framework, as follows. A \emph{connected dominating set} (\emph{CD set} for short) in a connected graph $G$ is a set $S$ of vertices of $G$ that is \emph{dominating}, that is, every vertex of $G$ is either in $S$ or has a neighbor in $S$, and \emph{connected}, that is, the subgraph of $G$ induced by $S$ is connected.
The ground set $U$ is the vertex set of a connected graph $G = (V,E)$, property $\pi$ is the property of being a connected dominating set in $G$, and $T$ is any interval unbounded from above.

\medskip
Our motivations for studying the notion of connected domination in the above threshold framework are twofold.
First, connected domination is one of the most basic of the many variants of domination, with applications in modeling wireless networks, see, e.g.,
~\cite{MR2986095,HHS2-98,MR1605684,MR2114596,MR2782462,MR2965384,MR2510358,MR2434962,MR2901109,MR2901149,MR2979503,Wu2001}.
The connected dominating set problem is the problem of finding a minimum connected dominating set in a given connected graph. This problem is \NP-hard (and hard to approximate) for general graphs and remains intractable even under significant restrictions, for instance, for the class of split graphs
(see Section~\ref{sec:WCDS}). On the other hand, as outlined above, the problem is polynomially solvable in the class of connected-domishold graphs equipped with weights as in the definition. This motivates the study of connected-domishold graphs. In particular, identification of subclasses of connected-domishold graphs might lead to new classes of graphs where the connected dominating set problem (or its weighted version) is polynomially solvable.

\smallskip
\begin{sloppypar}
Second, despite the growingly large variety of graph domination concepts studied in the literature (see, e.g.,~\cite{HHS2-98,MR1605684}), so far a relatively small number of ``threshold-like'' graph classes was studied with respect to notions of domination: the classes of domishold and equidominating graphs (corresponding to the usual domination), the class of equistable graphs (corresponding to independent domination), and the class of total domishold graphs (corresponding to total domination). These graph classes differ significantly with respect to their structural and algorithmic properties.
For instance, while the class of domishold graphs is a highly structured hereditary subclass of cographs, the classes of equistable and of total domishold graphs are not contained in any nontrivial hereditary class of graphs and are not structurally understood.\footnote{A class of graphs is said to be \emph{hereditary} if it is closed under vertex deletion.}
In particular, the class of total domishold graphs is as rich in its combinatorial structure as the class of threshold hypergraphs~\cite{MR3281177}, for which (despite being recognizable in polynomial time via linear programming~\cite{MR798011,MR2742439}) the existence of a ``purely combinatorial'' polynomial time recognition algorithm is an open problem~\cite{MR2742439}. These results, differences, and challenges provide further motivation for the study of structural and algorithmic properties of connected-domishold graphs.
\end{sloppypar}

\subsection{The definition}

\begin{definition}\label{def:main}
A graph $G=(V,E)$ is said to be \emph{connected-domishold} (CD for short) if there exists a pair $(w,t)$ where $w:V\to \mathbb{R}_+$ is a weight function and $t\in \mathbb{R}_+$ is a threshold such that for every subset $S\subseteq V$, $w(S):= \sum_{x\in S}w(x)\ge t$ if and only if $S$ is a connected dominating set in $G$. Such a pair $(w,t)$ will be referred to as a {\em connected-domishold (CD) structure} of $G$.\end{definition}

Note that if $G$ is disconnected, then $G$ does not have any connected dominating sets. However, for technical reasons, we consider disconnected graphs to be connected-domishold. (This convention is compatible with Definition~\ref{def:main}: setting $w(x) = 0$ for all $x\in V(G)$ and $t = 1$ yields a CD structure of a disconnected graph $G$.)
Note also that not every graph that is connected and domishold is connected-domishold: the $4$-vertex cycle is connected and domishold~(see, e.g.,~\cite{MR0491342}) but not CD.

\begin{example}\label{exa:cycle}
The $4$-cycle $C_4$ is not connected-domishold. Denoting its vertices by $v_1,v_2,v_3,v_4$ in a cyclic order, a subset $S\subseteq V(C_4)$ is CD if and only if it contains an edge. Therefore, if $(w,t)$ is a CD structure of $C_4$, then $w(v_i)+w(v_{i+1})\ge t$ for all $i\in \{1,2,3,4\}$ (indices modulo $4$), which implies $w(V(C_4))\ge 2t$. On the other hand, $w(v_1)+w(v_3)<t$ and $w(v_2)+w(v_4)<t$, implying $w(V(C_4))<2t$.
\end{example}

\begin{example}\label{exa:CompleteGraph}
The complete graph of order $n$ is connected-domishold. Indeed, any nonempty subset $S\subseteq V(K_n)$ is a connected dominating set of $K_n$,
and the pair $(w,1)$ where $w(x) = 1$ for all $x\in V(K_n)$ is a CD structure of $K_n$.
\end{example}

\subsection{Overview of results}

Our results can be divided into four interrelated parts and can be summarized as follows:

\medskip
\begin{enumerate}[1)]
  \item \begin{sloppypar}{\bf Characterizations in terms of derived hypergraphs (resp., derived Boolean functions);
  a necessary and a sufficient condition.}
\end{sloppypar}

\smallskip
  In a previous work~\cite[Proposition 4.1 and Theorem 4.5]{MR3281177}, total domishold graphs were characterized in terms of thresholdness of a derived hypergraph and a derived Boolean function. We give similar characterizations of connected-domishold graphs (Proposition~\ref{prop:c-dom-graphs}). The characterizations lead to a necessary and a sufficient condition for a graph to be connected-domishold, respectively, expressed in terms of properties of the derived hypergraph (equivalently: of the derived Boolean function; Corollary~\ref{prop:c-dom-graphs-ms-2-asummable}).

\medskip
  \item {\bf The case of split graphs. A characterization of threshold hypergraphs.}

  While the classes of connected-domishold and total domishold graphs are in general incomparable, we show that they coincide within the class of connected split graphs (Theorem~\ref{thm:split}). Building on this equivalence, we characterize threshold hypergraphs in terms of the connected-domisholdness property of a derived split graph (Theorem~\ref{thm:hyp-split}). We also give examples of connected split graphs showing that neither of the two conditions for a graph to be connected-domishold mentioned above (one necessary and one sufficient) characterizes this property.

\medskip
  \item {\bf The hereditary case.}

\begin{sloppypar}
  We observe that, contrary to the classes of threshold and domishold graphs, the class of connected-domishold graphs is not hereditary. This motivates the study of so-called \emph{hereditarily connected-domishold graphs}, defined as graphs every induced subgraph of which is connected-domishold. As our main result (Theorem~\ref{thm:characterizations}), we give several characterizations of the class of hereditarily connected-domishold graphs.
  The characterizations in terms of forbidden induced subgraphs implies that the class of hereditarily connected-domishold graphs is a subclass of the class of chordal graphs properly containing two well known classes of chordal graphs, the class of block graphs and the class of trivially perfect graphs.
\end{sloppypar}

\medskip
  \item {\bf Algorithmic aspects via vertex separators.}

 Finally, we build on all these results, together with some known results from the literature on connected dominating sets and minimal vertex separators in graphs, to study certain algorithmic aspects of the class of connected-domishold graphs and their hereditary variant. We identify a sufficient condition, capturing a large number of known graph classes, under which the CD property can be recognized efficiently (Theorem~\ref{thm:poly-cutsets}). We also show that the same condition, when applied to classes of connected-domishold graphs, results in classes of graphs for which the weighted connected dominating set  problem (which is \NP-hard even on split graphs) is polynomially solvable (Theorem~\ref{thm:poly}).
 This includes the classes of hereditarily connected-domishold graphs and $F_2$-free split graphs (see Fig.~\ref{fig:F2}), leading thus to new polynomially solvable cases of the problem.
\begin{figure}[!ht]
  \begin{center}
\includegraphics[width=0.2\linewidth]{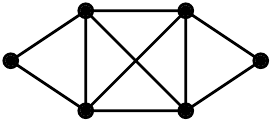}
  \end{center}
    \caption{Graph $F_2$.}
\label{fig:F2}
\end{figure}
\end{enumerate}

\medskip
\begin{sloppypar}
\noindent{\bf Structure of the paper.} In Section~\ref{sec:prelim}, we state the necessary definitions and preliminary results on graphs, hypergraphs, and Boolean functions. In Section~\ref{sec:hypergraphs}, we give characterizations of connected-domishold graphs in terms of thresholdness of derived hypergraphs and Boolean functions. Connected-domishold split graphs are studied in Section~\ref{sec:split}, where their relation to threshold hypergraphs is also observed. The main result of the paper, Theorem~\ref{thm:characterizations}, is stated in Section~\ref{sec:main}, where some of its consequences are also derived. Section~\ref{sec:algo} discusses some algorithmic aspects of connected-domishold graphs.
Our proof of Theorem~\ref{thm:characterizations} relies on a technical lemma, which is proved in Section~\ref{sec:proof}.
\end{sloppypar}

\section{Preliminaries}\label{sec:prelim}

\subsection{Graphs}

All graphs in this paper will be finite, simple and undirected. The \emph{(open) neighborhood} of a vertex $v$ is the set of vertices in a graph $G$ adjacent to $v$, denoted by $N_G(v)$ (or simply $N(v)$ if the graph is clear from the context); the \emph{closed neighborhood} of $v$ is denoted by $N_G[v]$ and defined as $N_G(v)\cup\{v\}$. The \emph{degree} of a vertex $v$ in a graph $G$ is the cardinality of its neighborhood. The complete graph, the path and the cycle of order $n$ are denoted by $K_n$, $P_n$ and $C_n$, respectively. A \emph{clique} in a graph is a subset of pairwise adjacent vertices, and an \emph{independent} (or \emph{stable}) set is a subset of pairwise non-adjacent vertices.
A \emph{universal vertex} in a graph $G$ is a vertex adjacent to all other vertices. For a set $S$ of vertices in a graph $G$, we denote by $G[S]$ the subgraph of $G$ induced by $S$. For a set $\mathcal{F}$ of graphs, we say that a graph is $\mathcal{F}$-free if it does not contain any induced subgraph isomorphic to a member of $\mathcal{F}$. Given a graph $G$, a vertex $v\in V(G)$, and a set $U\subseteq V(G)\setminus\{v\}$, we say that $v$ {\em dominates} $U$ if $v$ is adjacent to every vertex in $U$.

The main notion that will provide the link between threshold Boolean functions and hypergraphs is that of cutsets in graphs.
A \emph{cutset} in a graph $G$ is a set $S\subseteq V(G)$ such that $G-S$ is disconnected. A cutset is \emph{minimal} if it does not contain any other cutset. For a pair of disjoint vertex sets $A$ and $B$ in a graph $G$ such that no vertex in $A$ has a neighbor in $B$, an \emph{$A,B$-separator} is a set of vertices $S\subseteq V(G)\setminus (A\cup B)$ such that $A$ and $B$ are in different components of $G-S$. An $A,B$-separator is said to be \emph{minimal} if it does not contain any other $A,B$-separator.
When sets $A$ and $B$ are singletons, say $A = \{u\}$ and $B = \{v\}$, we will refer to a (minimal) $A,B$-separator simply as a \emph{(minimal) \hbox{$u,v$-separator}}. A \emph{minimal vertex separator} in $G$ is a minimal $u,v$-separator for some non-adjacent vertex pair $u$, $v$. Note that every minimal cutset of $G$ is a minimal vertex separator, but not vice versa. The minimal cutsets are exactly the minimal $u,v$-separators that do not contain any other $x,y$-separator. The connection between the CD graphs and the derived hypergraphs and Boolean functions will be developed in Section~\ref{sec:hypergraphs} using the following characterization of CD sets due to Kant\'e et al.~\cite{KLMN}.

\begin{proposition}[Kant\'e et al.~\cite{KLMN}]\label{prop:c-dom-ms}
In every graph $G$ that is not complete, a subset $D\subseteq V(G)$ is a CD set if and only if $D\cap S\neq \emptyset$ for every minimal cutset $S$ in $G$.
\end{proposition}

In other words, the CD sets of a graph $G$ are exactly the transversals of the hypergraph of the minimal cutsets of $G$
(see Section~\ref{sec:hypergraphs-prelim} and Definition~\ref{def:MCH} for definitions of these notions).

A graph $G$ is \emph{chordal} if it does not contain any induced cycle of order at least $4$, and \emph{split} if it has a \emph{split partition}, that is, a partition of its vertex set into a clique and an independent set. One of our proofs (the proof of Theorem~\ref{thm:characterizations}) will rely on the following property of chordal graphs.

\begin{lemma}[Kumar and Veni Madhavan~\cite{KumarM98}]\label{lem:vertices-dominating-cutsets-in-chordal-graphs}
If $S$ is a minimal cutset of a chordal graph $G$, then each connected component of $G - S$ has a vertex that is adjacent to all the vertices of $S$.
\end{lemma}

For graph theoretic notions not defined above, see, e.g.,~\cite{opac-b1096791}.

\subsection{Boolean functions}
Let $n$ be positive integer. Given two vectors $x,y\in \{0,1\}^n$, we write $x\le y$ if $x_i\le y_i$ for all $i\in [n] := \{1,\ldots, n\}$.
A Boolean function $f:\{0,1\}^n\to\{0,1\}$ is \emph{positive} (or: \emph{monotone}) if $f(x)\le f(y)$ holds for
every two vectors $x,y\in \{0,1\}^n$ such that $x\le y$. A \emph{literal} of $f$ is either a variable, $x_i$, or the negation of a variable, denoted by $\overline{x_i}$. An \emph{implicant} of a Boolean function $f$ is a conjunction $C$ of literals such that $f(x) = 1$ for all $x\in \{0,1\}^n$ for which $C$ takes value $1$ (we also say that $C$ \emph{implies} $f$). An implicant is said to be \emph{prime} if it is not implied by any other implicant. If $f$ is positive, then none of the variables appearing in any of its prime implicants appears negated. Every $n$-variable positive Boolean function $f$ can be expressed with its \emph{complete DNF (disjunctive normal form)}, defined as the disjunction of all prime implicants of $f$.

A positive Boolean function $f$ is said to be \emph{threshold} if there exist non-negative real weights $w=(w_1,\ldots, w_n)$ and a non-negative real number $t$ such that for every $x\in \{0,1\}^n$, $f(x) = 0$ if and only if $\sum_{i = 1}^nw_ix_i\le t$. Such a pair $(w,t)$ is called a \emph{separating structure} of $f$. Every threshold Boolean function admits an integral separating structure (see~\cite[Theorem 9.5]{MR2742439}). A positive Boolean function $f(x_1,\ldots, x_n)$ is threshold if and only if its \emph{dual function} $f^d(x) = \overline{f(\overline x)}$ is threshold~\cite{MR2742439}; moreover, if $(w_1,\ldots, w_n,t)$ is an integral separating structure of $f$, then $(w_1,\ldots, w_n,\sum_{i = 1}^nw_i-t-1)$ is a separating structure of $f^d$.

Threshold Boolean functions have been characterized in~\cite{MR0128319} and~\cite{conf/focs/Elgot60}, as follows. A \emph{false point} of $f$ is an input vector $x\in \{0,1\}^n$ such that $f(x) = 0$; a \emph{true point} is defined analogously. For $k\ge 2$, a positive Boolean function $f: \{0,1\}^n\to\{0,1\}$ is said to be \emph{$k$-summable} if, for some $r \in \{2,\ldots,k\}$,  there exist $r$ (not necessarily distinct) false points of $f$, say, $x^1, x^2,\ldots, x^r$, and $r$ (not necessarily distinct) true points of $f$, say $y^1, y^2,\ldots, y^r$, such that
$\sum_{i=1}^r x^i = \sum_{i=1}^r y^i$ (note that the sums are in $\mathbb{Z}^n$ and not in $\mathbb{Z}_2^n$, the $n$-dimensional vector space over GF($2$)). Function $f$ is said to be \emph{$k$-asummable} if it is not $k$-summable, and it is \emph{asummable} if it is $k$-asummable for all $k \geq 2$.

\begin{theorem}[Chow~\cite{MR0128319}, Elgot~\cite{conf/focs/Elgot60}, see also~{\cite[Theorem 9.14]{MR2742439}}]\label{thm:BF:characterization}
A positive Boolean function $f$ is threshold if and only if it is asummable.
\end{theorem}

The problem of determining whether a positive Boolean function given by its complete DNF is threshold is solvable in polynomial time, using dualization and linear programming (see~\cite{MR798011} and~\cite[Theorem 9.16]{MR2742439}). The algorithm tests if a polynomially sized derived linear program has a feasible solution, and in case of a yes instance, the solution found yields a separating structure of the given function. Using, e.g., Karmarkar's interior point method for linear programming~\cite{MR779900}, one can assure that a rational solution is found. This results in a rational separating structure, which can be easily turned into an integral one. We summarize this result as follows.

\begin{theorem}\label{thm:thr-Bf}
There exists a polynomial time algorithm for recognizing threshold Boolean functions given by the complete DNF.
In case of a yes instance, the algorithm also computes an integral separating structure of the given function.
\end{theorem}

\begin{remark}\label{rem:PCH}
The existence of a ``purely combinatorial'' polynomial time recognition algorithm for threshold Boolean functions (that is, one not relying on solving an auxiliary linear program) is an open problem~\cite{MR2742439}.
\end{remark}

\begin{sloppypar}
A similar approach as the one outlined above shows that every connected-domishold graph has an integral CD structure; we will often use this fact in the paper. For further background on Boolean functions, we refer to the comprehensive monograph by Crama and Hammer~\cite{MR2742439}.
\end{sloppypar}

\subsection{Hypergraphs}\label{sec:hypergraphs-prelim}

A \emph{hypergraph} is a pair $\mathcal{H} = (V,E)$ where $V$ is a finite set of {\em vertices} and $E$ is a set of subsets of $V$, called {\em hyperedges}~\cite{MR1013569}. When the vertex set or the hyperedge set of $\cal H$ will not be explicitly given, we will refer to them by
$V(\mathcal{H})$ and $E(\mathcal{H})$, respectively. A \emph{transversal} (or: \emph{hitting set}) of $\mathcal{H}$ is a set $S\subseteq V$ such that \hbox{$S\cap e\neq \emptyset$} for all $e\in E$. A hypergraph $\mathcal{H} = (V,E)$ is {\em threshold} if there exist a weight function $w:V\to \mathbb{R}_+$ and a threshold $t\in \mathbb{R}_+$ such that for all subsets $X\subseteq V$, it holds that $w(X)\le t$ if and only if $X$ contains no hyperedge of $\cal H$~\cite{MR2063679}. Such a pair $(w,t)$ is said to be a \emph{separating structure} of $\mathcal{H}$.

\begin{sloppypar}
To every hypergraph $\mathcal{H} = (V,E)$, we can naturally associate a positive Boolean function \hbox{$f_\mathcal{H}:\{0,1\}^{V}\to\{0,1\}$,} defined by the positive DNF expression
$$f_\mathcal{H}(x) = \bigvee_{e\in E} \bigwedge_{u\in e}x_u$$
for all $x\in \{0,1\}^V$.
Conversely, to every positive Boolean function $f:\{0,1\}^n\to \{0,1\}$ given by a positive  DNF $\phi = \bigvee_{j = 1}^m \bigwedge_{i\in C_j}x_i$, we can associate a hypergraph $\mathcal{H}(\phi) = (V,E)$ as follows: ${V} = [n]$ and ${E} = \{C_1,\ldots, C_m\}$.
It follows directly from the definitions that the thresholdness of hypergraphs and of Boolean functions are related as follows.
\end{sloppypar}

\begin{proposition}\label{prop:relation}
A hypergraph $\mathcal{H} = (V,E)$ is threshold if and only if the positive Boolean function $f_\mathcal{H}$ is threshold.
A positive Boolean function given by a positive DNF $\phi = \bigvee_{j = 1}^m \bigwedge_{i\in C_j}x_i$ is threshold if and only if
the hypergraph $\mathcal{H}(\phi)$ is threshold.
\end{proposition}

Applying Theorem~\ref{thm:BF:characterization} to the language of hypergraphs gives the following characterization of threshold hypergraphs.
For $k\geq 2$, a hypergraph $\mathcal{H}=(V,E)$ is said to be \emph{$k$-summable} if, for some $r \in \{2,\ldots,k\}$, there exist $r$ (not necessarily distinct) subsets $A_1,\dots,A_r$ of $V$ such that each $A_i$ contains a hyperedge of $\mathcal{H}$, and $r$ (not necessarily distinct) subsets $B_1,\dots,B_r$ of $V$ such that each $B_i$ does not contain a hyperedge of $\mathcal{H}$
and such that for every vertex $v\in V$, we have:

\begin{equation}\label{eq:k-summable hypergraph}
|\{i:v\in A_i\}| = |\{i: v\in B_i\}|.
\end{equation}
We say that a hypergraph $\mathcal{H}$ is \emph{$k$-asummable} if it is not $k$-summable and it is \emph{asummable} if it is $k$-asummable for all $k \geq 2$.

\begin{corollary}\label{cor:Threshold hyp by asummability}
A hypergraph $\mathcal{H}$ is threshold if and only if it is asummable.
\end{corollary}

A hypergraph $\mathcal{H} = (V,E)$ is said to be {\em Sperner} (or: a \emph{clutter}) if no hyperedge of $\cal H$ contains another hyperedge, that is, if for every two distinct hyperedges $e$ and $f$ of $\cal H$, it holds that $\min\{|e\setminus f|,|f\setminus e|\}\ge 1\,.$ Chiarelli and Milani\v c defined in~\cite{ChiMil13,MR3281177} the notion of \emph{dually Sperner hypergraphs} as the hypergraphs such that the inequality $\min\{|e\setminus f|,|f\setminus e|\}\le  1$ holds for every pair of distinct hyperedges $e$ and $f$ of $\cal H$. It was proved in~\cite{ChiMil13,MR3281177} that dually Sperner hypergraphs are threshold; they were applied in the characterizations of total domishold graphs and their hereditary variant. Boros et al.~introduced in~\cite{BGM2016+} the following restriction of dually Sperner hypergraphs.

\begin{definition}[Boros et al.~\cite{BGM2016+}]
A hypergraph $\mathcal{H} = (V,E)$ is said to be {\em $1$-Sperner} if for every two distinct hyperedges $e$ and $f$ of $\cal H$, it holds that
$\min\{|e\setminus f|,|f\setminus e|\}= 1\,.$
\end{definition}

Note that a hypergraph is $1$-Sperner if and only if it is both Sperner and dually Sperner. In particular,
for Sperner hypergraphs the notions of dually Sperner and $1$-Sperner hypergraphs coincide.
Since a hypergraph $\mathcal{H}$ is threshold if and only if the
Sperner hypergraph obtained from $\mathcal{H}$ by keeping only its inclusion-wise minimal hyperedges is threshold,
the fact that dually Sperner hypergraphs are threshold is equivalent to the following
fact, proved constructively by Boros et al.~in~\cite{BGM2016+} using a composition result for $1$-Sperner hypergraphs developed therein.

\begin{theorem}[Chiarelli and Milani\v c~\cite{MR3281177}, Boros et al.\cite{BGM2016+}]\label{lem:1-Sperner-threshold}
Every $1$-Sperner hypergraph is threshold.
\end{theorem}

\begin{sloppypar}
\section{Connected-domishold graphs via hypergraphs and Boolean functions}\label{sec:hypergraphs}
\end{sloppypar}

In a previous work~\cite[Proposition 4.1 and Theorem 4.5]{MR3281177}, total domishold graphs were characterized in terms of
thresholdness of a derived hypergraph and a derived Boolean function. In this section we give similar characterizations of connected-domishold graphs.

We first recall some relevant definitions and a result from~\cite{MR3281177}. A \emph{total dominating set} in a graph $G$ is a set $S\subseteq V(G)$ such that every vertex of $G$ has a neighbor in $S$. Note that only graphs without isolated vertices have total dominating sets.
A graph $G=(V,E)$ is said to be {\em total domishold} (TD for short) if there exists a pair $(w,t)$ where $w:V\to \mathbb{R}_+$ is a weight function and $t\in \mathbb{R}_+$ is a threshold such that for every subset $S\subseteq V$, $w(S):= \sum_{x\in S}w(x)\ge t$ if and only if $S$ is a total dominating set in $G$. A pair $(w,t)$ as above will be referred to as a {\em total domishold (TD) structure} of $G$.
The \emph{neighborhood hypergraph} of a graph $G$ is the hypergraph denoted by $\MNH(G)$
and defined as follows: the vertex set of $\MNH(G)$ is $V(G)$ and the hyperedge set consists precisely of the
\emph{minimal neighborhoods} in $G$, that is, of the inclusion-wise minimal sets in the family of neighborhoods $\{N(v):v\in V(G)\}$.\footnote{In~\cite{MR3281177}, the neighborhood hypergraph of $G$ was named \emph{reduced neighborhood hypergraph} (of $G$) and denoted
by $\mathcal{RN}(G)$. We changed the terminology in analogy with the term ``cutset hypergraph'', which will be introduced shortly.}
Note that a set $S\subseteq V(G)$ is a total dominating set in $G$ if and only if it is a transversal of $\MNH(G)$.

\medskip
\begin{proposition}[Chiarelli and Milani\v c~{\cite{MR3281177}}]\label{prop:t-dom-graphs}
For a graph $G=(V,E)$, the following are equivalent:
\begin{enumerate}
  \item $G$ is total domishold.
  \item Its neighborhood hypergraph $\MNH(G)$ is threshold.
\end{enumerate}
\end{proposition}

\begin{sloppypar}
The constructions of the derived hypergraph and the derived Boolean function used in our characterizations of connected-domishold graphs in terms of their thresholdness are specified by Definitions~\ref{def:MCH} and~\ref{def:MCH-Bf}.
\end{sloppypar}

\begin{definition}\label{def:MCH}
Given a graph $G$, the {\em cutset hypergraph} of $G$ is the hypergraph $\MCH(G)$ with vertex set $V(G)$ whose hyperedges are precisely the minimal cutsets in $G$.
\end{definition}

Given a finite non-empty set $V$, we denote by $\{0,1\}^V$ the set of all binary vectors with coordinates indexed by $V$.
Given a graph $G = (V,E)$ and a binary vector $x\in \{0,1\}^V$, its \emph{support set} is the set denoted by
$S(x)$ and defined by $S(x)=\{v\in V\,:\,x_v = 1\}$. In the following definition, we associate a Boolean function to a given $n$-vertex graph $G$.
In order to avoid fixing a bijection between its vertex set and the set $[n]$, we will consider the corresponding Boolean function as defined on the set
$\{0,1\}^{V}$, where $V = V(G)$. Accordingly, a separating structure of such a Boolean function can be seen as a pair
$(w,t)$ where $w:V\to \mathbb{R}^+$ and $t\in \mathbb{R}^+$ such that for every $x\in \{0,1\}^V$, we have
$f(x) = 0$ if and only if $\sum_{v\in S(x)}w(v)\le t$.

\begin{definition}\label{def:MCH-Bf}
Given a graph $G = (V,E)$, its {\em cutset function} is the positive Boolean function \hbox{$f_G^{\textit{cut}}:\{0,1\}^{V}\to\{0,1\}$} that takes value $1$ precisely on vectors $x\in \{0,1\}^V$ whose support set contains some minimal cutset of $G$.
\end{definition}

\begin{sloppypar}
The announced characterizations of connected-domishold graphs in terms of their cutset hypergraphs and cutset functions are given in the following proposition. The proof is based on two ingredients: the characterization of the connected dominating sets of a given (non-complete) graph given by Proposition~\ref{prop:c-dom-ms} and the fact that threshold Boolean functions are closed under dualization.
\end{sloppypar}

\begin{proposition}\label{prop:c-dom-graphs}
For a graph $G=(V,E)$, the following are equivalent:
\begin{enumerate}
  \item $G$ is connected-domishold.
  \item Its cutset hypergraph $\MCH(G)$ is threshold.
  \item Its cutset function $f_G^{\textit{cut}}$ is threshold.
\end{enumerate}
Moreover, if $G$ is not a complete graph and $(w,t)$ is an integral separating structure of $f_G^{\textit{cut}}$ or of $\MCH(G)$, then
$(w,w(V)-t)$ is a CD structure of~$G$.
\end{proposition}

\begin{proof}
We consider two cases, depending on whether $G$ is a complete graph or not.

\begin{thmcase}
$G$ is complete.

In this case all the three statements hold. Recall that every complete graph is CD (see Example~\ref{exa:CompleteGraph}). Since complete graphs have no cutsets, the set of hyperedges of the cutset hypergraph $\MCH(G)$ is empty. Hence the hypergraph $\MCH(G)$ is threshold.
The absence of (minimal) cutsets also implies that the cutset function $f_G^{\textit{cut}}$ is constantly equal to $0$
and hence threshold. \end{thmcase}

\begin{thmcase}
$G$ is not complete.

First we will show the equivalence between statements $1$ and $3$. Since a positive Boolean function $f$ is threshold if and only if its dual function $f^d(x) = \overline{f(\overline x)}$ is threshold, it suffices to argue that
$G$ is connected-domishold if and only if $(f_G^{\textit{cut}})^d$ is threshold.

We claim that for every $x\in \{0,1\}^{V}$, we have $(f_G^{\textit{cut}})^d(x) = 1$ if and only if
$S(x)$, the support set of $x$, is a connected dominating set of $G$. Let $x\in \{0,1\}^{V}$ and let $S$ be the support set of $x$.
By definition, $(f_G^{\textit{cut}})^d(x) = 1$ if and only if $f_G^{\textit{cut}}(\overline x) = 0$,
which is the case if and only if $V\setminus S$ does not contain any minimal cutset of $G$. This is in turn equivalent to the condition that $S$ is a transversal of the cutset hypergraph of $G$, and, by Proposition~\ref{prop:c-dom-ms}, to the condition that $S$ is a connected dominating set of $G$.
Therefore, $(f_G^{\textit{cut}})^d(x) = 1$ if and only if $S$ is a connected dominating set of $G$, as claimed.

Now, if $G$ is connected-domishold, then it has an integral connected-domishold structure, say $(w,t)$, and
$(w,t-1)$ is a separating structure of the dual function $(f_G^{\textit{cut}})^d$, which implies that
$(f_G^{\textit{cut}})^d$ is threshold. Conversely, if the dual function is threshold, with an integral separating structure $(w,t)$,
then $(w,t+1)$ is a connected-domishold structure of $G$. This establishes the equivalence between statements $1$ and $3$.

Next, we show the equivalence between statements $2$ and $3$. Note that the complete DNF of $f_G^{\textit{cut}}$, the cutset function of $G$, is given by the expression $\bigvee_{S\in \mathcal{C}(G)}\bigwedge_{u\in S}x_u$.
It now follows directly from the definitions of threshold Boolean functions and threshold hypergraphs that function $f_G^{\textit{cut}}(x)$ is threshold if and only if cutset hypergraph $\MCH(G)$ is threshold.

Finally, if $(w,t)$ is an integral separating structure of $f_G^{\textit{cut}}$, then
$(w,w(V)-t-1)$ is a separating structure of $(f_G^{\textit{cut}})^d$ and hence
$(w,w(V)-t)$ is a connected-domishold structure of $G$.
\end{thmcase}
\end{proof}

Recall that every $1$-Sperner hypergraph is threshold (Theorem~\ref{lem:1-Sperner-threshold}) and every threshold hypergraph is asummable (Corollary~\ref{cor:Threshold hyp by asummability}). Thus, in particular, every threshold hypergraph is \hbox{$2$-asummable}.
Applying these relations to the specific case of the minimal cutset  hypergraphs, Proposition~\ref{prop:c-dom-graphs} leads to the following.

\begin{corollary}\label{prop:c-dom-graphs-ms-2-asummable}
For every graph $G$, the following holds:
\begin{enumerate}
  \item If the cutset hypergraph $\MCH(G)$ is $1$-Sperner, then $G$ is connected-domishold.
  \item If $G$ is connected-domishold, then its cutset hypergraph $\MCH(G)$ is $2$-asummable.
\end{enumerate}
\end{corollary}

We will show in Section~\ref{sec:examples} that neither of the two statements in Corollary~\ref{prop:c-dom-graphs-ms-2-asummable} can be reversed. On the other hand, we will prove in Section~\ref{sec:main} that all the three properties become equivalent in the hereditary setting.

\section{Connected-domishold split graphs}\label{sec:split}

The following examples show that for general connected graphs, the CD and TD properties are incomparable:
\begin{itemize}
  \item The path $P_6$ is connected-domishold (it has a unique minimal connected dominating set, formed by its internal vertices)
but it is not total domishold (see, e.g.,~\cite{MR3281177}).
  \item The graph in Fig.~\ref{fig:example} is TD but not CD.

\begin{figure}[!ht]
  \begin{center}
\includegraphics[width=0.35\linewidth]{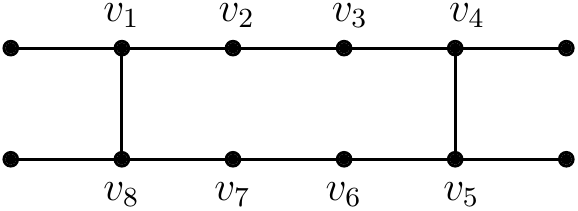}
  \end{center}
  \caption{A TD graph that is not CD.}
\label{fig:example}
\end{figure}
The graph is total domishold: it has a unique minimal total dominating set, namely $\{v_1,v_4,v_5,v_8\}$.
On the other hand, the graph is not connected-domishold. This can be shown by observing that its cutset hypergraph is not $2$-asummable and applying Corollary~\ref{prop:c-dom-graphs-ms-2-asummable}.
To see that the cutset hypergraph of $G$ is $2$-summable, note that
condition~\eqref{eq:k-summable hypergraph} is satisfied if we take $k = r = 2$ and
$A_1 = \{v_2,v_7\}$, $A_2 = \{v_3,v_6\}$, $B_1 = \{v_2,v_3\}$, and $B_2 = \{v_6,v_7\}$.
\end{itemize}

Interestingly, we will show in Section~\ref{sec:main} that if the CD and TD properties are required also for all induced subgraphs, then
the corresponding graph classes become comparable (see Corollary~\ref{cor:inclusion}). In the rest of this section, we will prove that the two properties coincide in the class of connected split graphs and examine some consequences of this result. Recall that a graph is split if and only if its vertex set has a partition into a clique and an independent set. Foldes and Hammer characterized split graphs as exactly the graphs that are $\{2K_2,C_4,C_5\}$-free~\cite{MR0505860}. In particular, this implies that a split graph can be disconnected only if it has an isolated vertex.

\begin{lemma}\label{lem:universal}
Let $G$ be a connected graph and let $G'$ be the graph obtained from $G$ by adding to it a universal vertex.
Then, $G$ is connected-domishold if and only if $G'$ is connected-domishold.
\end{lemma}

\begin{proof}
Let $V(G') = V(G)\cup \{u\}$, where $u$ is the added vertex.
Suppose that $G$ is connected-domishold and let $(w,t)$ be a CD structure of $G$.
Since the set of connected dominating sets of $G'$ consists of all connected dominating sets of $G$ together with all
subsets of $V(G')$ containing $u$, we can obtain a
CD structure, say $(w',t')$, of $G'$ by setting $w'(x) = w(x)$ for all $x\in V(G)$, $w'(u) = t$, and $t' = t$.
Therefore, $G'$ is connected-domishold.

Conversely, if $(w',t')$ is a CD structure of $G'$, then $(w,t)$ where $t = t'$ and $w$ is the restriction of $w'$ to $V(G)$
is a CD structure of $G$. This is because a set $X\subseteq V(G)$ is a connected dominating set of $G$ if and only if it is a connected dominating set of $G'$. Therefore, if $G'$ is connected-domishold then so is $G$.
\end{proof}

Recall that given a connected graph $G$, we denote by $\MCH(G)$ (resp., $\MNH(G)$) its cutset (resp., neighborhood) hypergraph.

\begin{lemma}\label{lem:split}
Let $G$ be a connected split graph without universal vertices. Then $\MCH(G) = \MNH(G)$.
\end{lemma}

\begin{proof}
Fix a split partition of $V(G)$, say $V(G) = K\cup I$ where $K$ is a clique, $I$ is an independent set, and $K\cap I = \emptyset$.
Clearly, the hypergraphs $\MCH(G)$ and $\MNH(G)$ have the same vertex set. We show that the hyperedge sets are also the same in two steps.

First, we show that $E(\MCH(G))\subseteq E(\MNH(G))$, that is, that every minimal cutset is a minimal neighborhood.
To this end, it suffices to show that every minimal cutset $S$ in $G$ \emph{is} a neighborhood, that is, a set of the form $S = N(v)$ for some $v\in V(G)$. This is indeed enough, because if a minimal cutset $S$ in $G$ satisfies $S = N(v)$ for some $v\in V(G)$, but $N(v)$ properly contains some other neighborhood, say $N(u)$, then the fact that $N(u)$ is a cutset in $G$ (for instance, it is a $u,v$-separator) would imply that $S$ is not a minimal cutset.

Let $S$ be a minimal cutset in $G$. Then, $S$ is a minimal $u,v$-separator for some non-adjacent vertex pair $u$, $v$; in particular,
$S\subseteq V(G)\setminus\{u,v\}$. We claim that $N(u)\subseteq S$ or $N(v)\subseteq S$.
Suppose that this is not the case. Then, there exist a neighbor of $u$, say $u'$, such that $u'\not\in S$, and
a neighbor of $v$, say $v'$, such that $v'\not\in S$. Since $\{u,v,u',v'\}\subseteq V(G)\setminus S$ and $u$ and $v$ are in different components of $G-S$,
vertices $u'$ and $v'$ are distinct and non-adjacent. Thus, at least one of $u'$ and $v'$, say $u'$, is in $I$.
This implies that $u\in K$ and therefore $v\in I$, which implies that $v'\in K$ and hence $(u,v',v)$ is a $u,v$-path in $G-S$, a contradiction.
This shows that $N(u)\subseteq S$ or $N(v)\subseteq S$, as claimed. Since each of $N(u)$ and $N(v)$ is a $u,v$-separator, the fact that $S$ is a minimal $u,v$-separator implies that $S\in \{N(u),N(v)\}$. This completes the proof of the inclusion $E(\MCH(G))\subseteq E(\MNH(G))$.

It remains to show that $E(\MNH(G))\subseteq E(\MCH(G))$. Let $S$ be a minimal neighborhood in $G$. Then $S = N(v)$ for some $v\in V(G)$.
Since $v$ is not universal, the set $V(G)\setminus N[v]$ is non-empty. Therefore $S$ is a $v,w$-separator for any $w\in V(G)\setminus N[v]$; in particular, $S$ is a cutset in $G$. Suppose for a contradiction that $S$ is not a minimal cutset in $G$.
Then $S$ properly contains some minimal cutset, say $S'$, in $G$. By the first part of the proof, $S'$ is of the form
$S' = N(z)$ for some $z\in V(G)$. However, since $N(z)$ is a neighborhood properly contained in $S = N(v)$, this contradicts the fact that
$S$ is a minimal neighborhood.
\end{proof}

\begin{theorem}\label{thm:split}
A connected split graph is connected-domishold if and only if it is total domishold.
\end{theorem}

\begin{proof}
If $G$ is complete, then $G$ is both connected-domishold and total domishold.
So we may assume that $G$ is not complete.
More generally, we show next that we may assume that $G$ does not have any universal vertices.
Suppose that $G$ has a universal vertex, say $u$, and let $G' = G-u$.
By~\cite[Proposition 3.3]{MR3281177}, $G$ is TD if and only if $G'$ is TD.
If $G'$ is not connected, then $\{u\}$ is the only minimal connected dominating set of $G$
and hence $G$ is connected-domishold in this case. Furthermore, $G$ is also total domishold:
since $G'$ is a disconnected $2K_2$-free graph, $G'$ has an isolated vertex.
Therefore, by~\cite{MR3281177}, $G'$ is TD, and hence so is $G$.
If $G'$ is connected, then by Lemma~\ref{lem:universal}, $G$ is CD if and only if $G'$ is CD.
Therefore, the problem of verifying whether the CD and the TD properties are equivalent for $G$ reduces to the same problem for $G'$.
An iterative application of the above argument eventually reduces the graph to either a graph
where both properties hold or to a connected graph without universal vertices.

Now, let $G$ be a connected split graph without universal vertices.
By Proposition~\ref{prop:c-dom-graphs}, $G$ is connected-domishold if and only if
its cutset hypergraph $\MCH(G)$ is threshold.
By Proposition~\ref{prop:t-dom-graphs}, $G$ is total domishold if and only if
its neighborhood hypergraph $\MNH(G)$ is threshold.
Therefore, to prove the theorem it suffices to show that $\MCH(G) = \MNH(G)$. But this was established in Lemma~\ref{lem:split}.
\end{proof}

Theorem~\ref{thm:split} implies another relation between connected-domishold (split) graphs and threshold hypergraphs,
one that in a sense reverses the one stated in Proposition~\ref{prop:c-dom-graphs}.
Given a hypergraph $\mathcal{H} = (V,E)$, the \emph{split-incidence graph} of $\cal H$ (see, e.g.,~\cite{KLMN}) is the split graph $G$ such that $V(G) = {V}\cup E$, $V$ is a clique, $E$ is an independent set, and $v\in V$ is adjacent to $e\in E$ if and only if $v\in e$.

\begin{theorem}\label{thm:hyp-split}
Let $\mathcal{H} = (V,E)$ be a hypergraph with $\emptyset\not\in E$. Then $\mathcal{H}$ is threshold if and only if its split-incidence graph is connected-domishold.
\end{theorem}

\begin{proof}
Since $\emptyset\not\in E$, the split-incidence graph of $\cal H$ is connected. It was shown in~\cite{MR3281177} that a hypergraph is threshold if and only if its split-incidence graph is total domishold. The statement of the theorem now follows from Theorem~\ref{thm:split}.
\end{proof}

It might be worth pointing out that in view of Remark~\ref{rem:PCH} and Theorem~\ref{thm:hyp-split}, it is an open problem of whether there is a ``purely combinatorial'' polynomial time algorithm for recognizing connected-domishold split graphs. Further issues regarding the recognition problem of CD graphs will be discussed in Section~\ref{sec:recognition}.

\subsection{Examples related to Corollary~\ref{prop:c-dom-graphs-ms-2-asummable}}\label{sec:examples}

We now show that neither of the two statements in Corollary~\ref{prop:c-dom-graphs-ms-2-asummable} can be reversed.
First we exhibit an infinite family of CD split graphs
whose cutset hypergraphs are not $1$-Sperner.

\begin{example}
Let $n\ge 4$ and let $G = K_n^*$ be the graph obtained from the complete graph $K_n$ by gluing a triangle on every edge.
Formally, $V(G) = \{u_1,\ldots, u_n\}\cup \{v_{ij}: 1\le i<j\le n\}$ and
$E(G) = \{u_iu_j: 1\le i<j\le n\}\cup \{u_iv_{jk}\mid 1\leq j<k\leq n \text{ and } i\in \{j,k\}\}$.
The graph $G$ is a CD graph: setting

$$w(x)=\left\{
                                              \begin{array}{ll}
                                                1, & \hbox{if $x \in \{u_1,\ldots, u_n\}$;} \\
                                                0, & \hbox{otherwise.}
                                              \end{array}
                                            \right.$$
and $t=n-1$ results in a CD structure of $G$.
On the other hand, the cutset hypergraph of $G$ is not $1$-Sperner.
Since every pair of the form $\{u_i,u_j\}$ with $1\le i<j\le n$
is a minimal cutset of $G$, the cutset hypergraph contains
$\{u_1,u_2\}$ and $\{u_3,u_4\}$ as hyperedges and is therefore not $1$-Sperner.
\end{example}

Next, we argue that there exists a split graph $G$
whose cutset hypergraph is $2$-asummable but $G$ is not CD. As observed already in~\cite{MR3281177}, the fact that not every $2$-asummable positive Boolean function is threshold can be used to construct split graphs $G$ such that $\MNH(G)$ is $2$-asummable and $G$ is not total domishold.
The existence of split graphs with claimed properties now follows from
Theorem~\ref{thm:split} and its proof. For the sake of self-containment, we describe an example of such a construction in Appendix.

\section{The hereditary case}\label{sec:main}

\begin{sloppypar}
In this section we present the main result of this paper, Theorem~\ref{thm:characterizations}, which gives several characterizations of graphs all induced subgraphs of which are CD, and derive some of its consequences. The proof of Theorem~\ref{thm:characterizations}
relies on a technical lemma about chordal graphs, which will be proved in Section~\ref{sec:proof}.
\end{sloppypar}

We start with an example showing that, contrary to the classes of threshold and domishold graphs, the class of connected-domishold graphs is not hereditary. We assume notation from Example~\ref{exa:cycle}.

\begin{example}
The graph $G$ obtained from $C_4$ by adding to it a new vertex, say $v_5$, and making it adjacent exactly to one vertex of the $C_4$, say to $v_4$, is CD: the (inclusion-wise) minimal CD sets of $G$ are $\{v_1,v_4\}$ and $\{v_3,v_4\}$, hence a CD structure of $G$ is given by
$w(v_2) = w(v_5) = 0$, $w(v_1) = w(v_3) = 1$, $w(v_4) = 2$, and $t = 3$.
\end{example}

This motivates the following definition.

\begin{sloppypar}
\begin{definition}
A graph $G$ is said to be {\em hereditarily connected-domishold} ({\em hereditarily CD} for short) if every induced subgraph of $G$ is connected-domishold.
\end{definition}
\end{sloppypar}

In order to state the technical lemma to be used in the proof of Theorem~\ref{thm:characterizations}, we need the following notation. A {\it diamond} is a graph obtained from $K_4$ by deleting an edge. Given a diamond $D$, we will refer to its vertices of degree two as its {\em tips} and denote them as $t$ and $t'$, and to its vertices of degree three as its {\em centers} and denote them as $c$ and $c'$. The respective vertex sets will be denoted by $T$ and $C$. Similar notation will be used for diamonds denoted by $D_1$ or $D_2$.

\begin{lemma}[Diamond Lemma]\label{diamondLemma}
Let $G$ be a connected chordal graph. Suppose that $G$ contains two induced diamonds $D_1=(V_1,E_1)$ and $D_2=(V_2,E_2)$ such that:
\begin{enumerate}[(i)]
  \item $C_1\cap C_2=\emptyset$.\label{ppt1}
  \item If no vertex in $C_1$ is adjacent to a vertex in $C_2$, then every minimal $C_1,C_2$-separator in $G$ is of size one.\label{ppt4}
  \item For each $j\in\{1,2\}$ the tips (i.e., $t_j,t_j'$) of $D_j$ belong to different components of $G-C_j$.\label{ppt2}
  \item For $j\in\{1,2\}$ every component of $G-C_j$ has a vertex that dominates $C_j$.\label{ppt3}
\end{enumerate}
Then $G$ has an induced subgraph isomorphic to $F_1,F_2,$ or $H_i$ for some $i\geq 1$, where the graphs $F_1$, $F_2$, and a general member of the family $\{H_i\}$ are depicted in Fig.~\ref{fig:forbidden-induced-subgraphs}.
\end{lemma}

\begin{figure}[!ht]
  \begin{center}
\includegraphics[width=0.8\linewidth]{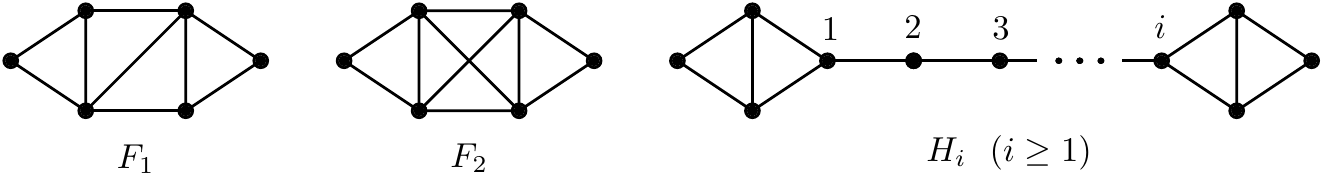}
  \end{center}
  \caption{Graphs $F_1$, $F_2$, and $H_i$.}
\label{fig:forbidden-induced-subgraphs}
\end{figure}

The proof of Lemma~\ref{diamondLemma} is postponed to Section~\ref{sec:proof}.

\begin{theorem}\label{thm:characterizations}
For every graph $G$, the following are equivalent:
\begin{enumerate}
  \item\label{item1} $G$ is hereditarily connected-domishold.
  \item\label{item3} The cutset hypergraph of every induced subgraph of $G$ is $1$-Sperner.
  \item\label{item5}  The cutset hypergraph of every induced subgraph of $G$ is threshold.
  \item\label{item2}  The cutset hypergraph of every induced subgraph of $G$ is $2$-asummable.
  \item\label{item4}  $G$ is an $\{F_1,F_2,H_1,H_2,\ldots\}$-free chordal graph.
\end{enumerate}
\end{theorem}

\begin{proof}
The equivalence between items $\ref{item1}$ and $\ref{item5}$ follows from Proposition~\ref{prop:c-dom-graphs}.

The implications $\ref{item3}\Rightarrow \ref{item1}\Rightarrow \ref{item2}$
follow from Corollary~\ref{prop:c-dom-graphs-ms-2-asummable}.

\begin{sloppypar}
For the implication $\ref{item2}\Rightarrow \ref{item4}$, we only need to verify that the cutset hypergraph of none of the graphs in the set
\hbox{$\mathcal{F} :=\{C_k\,:\,k\ge 4\}\cup \{F_1,F_2\}\cup \{H_i\,:\,i\ge 1\}$} is $2$-asummable. Let $F\in \mathcal{F}$. Suppose first that $F$ is a cycle $C_k$ for some $k\ge 4$, let $u_1, u_2, u_3, u_4$ be four consecutive vertices on the cycle.
Let $A_1 = \{u_1,u_3\}$, $A_2 = \{u_2,u_4\}$, $B_1 = \{u_1,u_2\}$ and $B_2 = \{u_3,u_4\}$.
Then, $A_1$ and $A_2$ are minimal cutsets of $F$ and thus hyperedges of the hypergraph $\MCH(F)$, while $B_1$ and $B_2$ do not contain any minimal cutset of $F$ and are consequently independent sets in the hypergraph $\MCH(F)$. Since the sets $A_1,A_2,B_1$ and $B_2$ satisfy condition~(\ref{eq:k-summable hypergraph}), this implies that the hypergraph $\MCH(F)$ is $2$-summable.
If $F \in \{F_1, F_2\}\cup\{H_i\,:\,i\ge 1\}$, then let $a$ and $b$ be the two vertices of degree~$2$ in $F$,
let $N(a) = \{a_1,a_2\}$, $N(b) = \{b_1,b_2\}$, let $A_1 = N(a)$, $A_2 = N(b)$, $B_1 =\{a_1,b_1\}$ and $B_2 = \{a_2,b_2\}$.
The rest of the proof is the same as above.
\end{sloppypar}

It remains to show the implication $\ref{item4}\Rightarrow \ref{item3}$. Suppose for a contradiction that the implication fails and let $G$ be a minimal counterexample. That is, $G$ is an $\{F_1,F_2, H_1,H_2,\ldots\}$-free chordal graph such that its cutset hypergraph is not $1$-Sperner, but the
cutset hypergraph of every $\{F_1,F_2, H_1,H_2,\ldots\}$-free chordal graph with fewer vertices than $G$ is $1$-Sperner.
Since $\MCH(G)$ is not $1$-Sperner, $G$ has two minimal cutsets, say $S$ and $S'$, such that $\min\{|S\setminus S'|,|S'\setminus S|\} \geq 2$. The minimality of $G$ implies that the empty set is not a minimal cutset, hence $G$ is connected. Furthermore, the minimality ensures that $S$ and $S'$ are disjoint sets (otherwise one can remove $S \cap S'$ from $G$ and have a smaller counterexample). Thus, $\min\{|S|,|S'|\} \geq 2$. The minimality also ensures that $|S| = |S'| = 2$. Indeed, removing a third vertex $z$, if present, from $S$ does not affect the minimal cutset status of $S$. Since every minimal cutset in a chordal graph is a clique~\cite{MR0130190}, removing a third vertex $z$, if present, from $S$ will also not affect the minimal cutset status of $S'$ since the entire $S$ (which is a clique) is present in one component of $G-S'$.

The minimality also ensures that if there are no edges between $S$ and $S'$, then every minimal $S,S'$-separator $T$ is of size one. Indeed, if this is not the case, then $|T|\ge 2$ since $G$ is connected. Let $X$ be a component of $G-S$ containing $S'$ and let $Y$ be any other component of $G-S$. The fact that $T$ separates $S$ from $S'$ implies that $T$ contains all vertices in $N(S)\cap V(X)$, which is a non-empty set due to the minimality of $S$. Since $T$ is a minimal cutset in a chordal graph, it is a clique; in particular, it is fully contained in $X$. However, this implies that the sets $S'$ and $T$ are minimal cutsets in the graph $G-V(Y)$ such that $\min\{|S'\setminus T|,|T\setminus S'|\} \geq 2$, contrary to the minimality of $G$.

Let $X$, $Y$ be two distinct components of $G-S$ and $X'$, $Y'$ two distinct components of $G-S'$. By Lemma~\ref{lem:vertices-dominating-cutsets-in-chordal-graphs}, there exist vertices $x\in X$ and $y\in Y$ such that each of $x$ and $y$ dominates $S$ and $x' \in X'$ and $y' \in Y'$ such that each of $x'$ and $y'$ dominates $S'$. Let $D_1$ be the subgraph of $G$ induced by $S\cup \{x,y\}$ and let $D_2$ be the subgraph of $G$ induced by $S'\cup \{x',y'\}$. The definitions of $D_1$ and $D_2$ and Lemma~\ref{lem:vertices-dominating-cutsets-in-chordal-graphs} imply that $D_1$ and $D_2$ are two induced diamonds in $G$ satisfying the hypotheses of the Diamond Lemma (Lemma~\ref{diamondLemma}). Consequently, $G$ has an induced subgraph isomorphic to $F_1,F_2,$ or $H_i$ for some $i\geq 1$, a contradiction. This completes the proof of the theorem.
\end{proof}

In the rest of this section, we examine some of the consequences of the forbidden induced subgraph characterization of hereditarily CD graphs given by Theorem~\ref{thm:characterizations}.  The \emph{kite} (also known as the \emph{co-fork} or the \emph{co-chair}) is the graph depicted in Fig.~\ref{fig:kite}.

\begin{figure}[!ht]
  \begin{center}
\includegraphics[width=0.17\linewidth]{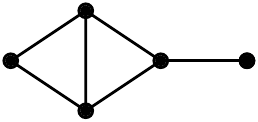}
  \end{center}
  \caption{The kite.}\label{fig:kite}
\end{figure}

The equivalence between items $\ref{item1}$ and $\ref{item4} $ in Theorem~\ref{thm:characterizations} implies that the class of hereditarily CD graphs is a proper generalization of the class of kite-free chordal graphs.


\begin{corollary}\label{cor:kite}
Every kite-free chordal graph is hereditarily CD.
\end{corollary}

Corollary~\ref{cor:kite} further implies that the class of hereditarily CD graphs generalizes two well known classes of chordal graphs,
the class of block graphs and the class of trivially perfect graphs. A graph is said to be a \emph{block graph} if every block (maximal connected subgraph without cut vertices) of it is complete. The block graphs are well known to coincide with the diamond-free chordal graphs. A graph $G$ is said to be \emph{trivially perfect}~\cite{MR522739} if for every induced subgraph $H$ of $G$, it holds $\alpha(H) = |\mathcal{K}(H)|$, where
$\alpha(H)$ denotes the \emph{independence number} of $H$ (that is, the maximum size of an independent set in $H$) and $\mathcal{K}(H)$
denotes the set of all maximal cliques of $H$. Trivially perfect graphs coincide with the so-called \emph{quasi-threshold graphs}~\cite{MR1402343}, and are exactly the $\{P_4, C_4\}$-free graphs~\cite{MR522739}.

\begin{corollary}
Every block graph is hereditarily CD. Every trivially perfect graph is hereditarily~CD.
\end{corollary}

Another class of graphs contained in the class of hereditarily CD graphs is the class of graphs defined similarly as the
hereditarily CD graphs but with respect to total dominating sets. These so-called \emph{hereditarily total domishold graphs}
(abbreviated \emph{hereditarily TD graphs}) were studied in~\cite{MR3281177}, where characterizations analogous to those given by Theorem~\ref{thm:characterizations} were obtained, including the following characterization in terms of forbidden induced subgraphs.

\begin{theorem}[Chiarelli and Milani\v c~\cite{MR3281177}]\label{thm:HTD}
For every graph $G$, the following are equivalent:
\begin{enumerate}
  \item $G$ is hereditarily total domishold.
  \item No induced subgraph of $G$ is isomorphic to a graph in Fig.~\ref{fig:forbidden-induced-subgraphs-for-HTD}.
\end{enumerate}
\end{theorem}

\begin{figure}[!ht]
  \begin{center}
\includegraphics[width=\linewidth]{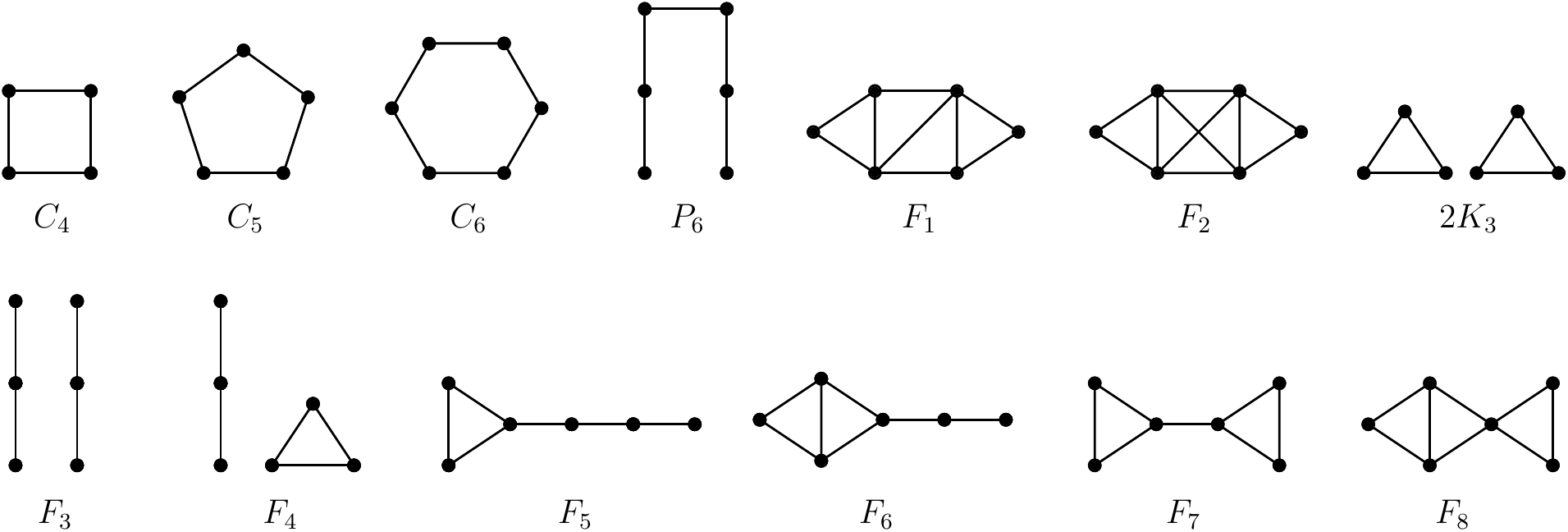}
  \end{center}
  \caption{The set of forbidden induced subgraphs for the class of hereditarily total domishold graphs.}
\label{fig:forbidden-induced-subgraphs-for-HTD}
\end{figure}

Theorems~\ref{thm:characterizations} and~\ref{thm:HTD} imply the following.

\begin{corollary}\label{cor:inclusion}
Every hereditarily TD graph is hereditarily CD.
\end{corollary}

\begin{proof}
It suffices to verify that each of the forbidden induced subgraphs for the class of hereditarily connected-domishold graphs contains one of the graphs from Fig.~\ref{fig:forbidden-induced-subgraphs-for-HTD} as induced subgraph.
A cycle $C_k$ with $k\ge 4$ contains (or is equal to) one of $C_4, C_5, C_6, P_6$.
The graphs $F_1$ and $F_2$ are contained in both sets of forbidden induced subgraphs.
Finally, each graph of the form $H_i$ where $i\ge 1$ contains $2K_3$ as induced subgraph.
\end{proof}

Since a graph is split if and only if it is $\{2K_2,C_4,C_5\}$-free and each of the forbidden induced subgraphs for the class of hereditarily total domishold graphs other than $F_2$ contains either $2K_2$, $C_4$, or $C_5$ as induced subgraph, Corollary~\ref{cor:inclusion} implies the following.

\begin{corollary}\label{cor:split}
Every $F_2$-free split graph is hereditarily CD.
\end{corollary}

Fig.~\ref{fig:Hasse} shows a Hasse diagram depicting the inclusion relations among the class of hereditarily connected-domishold graphs and several other,  well studied graph classes. All definitions of graph classes depicted in Fig.~\ref{fig:Hasse} and the relations between them can be found in~\cite{graphclasses}, with the exception of hereditarily CD and hereditarily TD graphs. The fact that every co-domishold graph is hereditarily TD and that every hereditarily TD graph is $(1,2)$-polar chordal was proved in~\cite{MR3281177}. The remaining inclusion and non-inclusion relations can be easily verified using the forbidden induced subgraph characterizations of the depicted graph classes, see~\cite{MR1686154,graphclasses,MR2063679}.

\begin{figure}[!ht]
  \begin{center}
\includegraphics[width=0.8\linewidth]{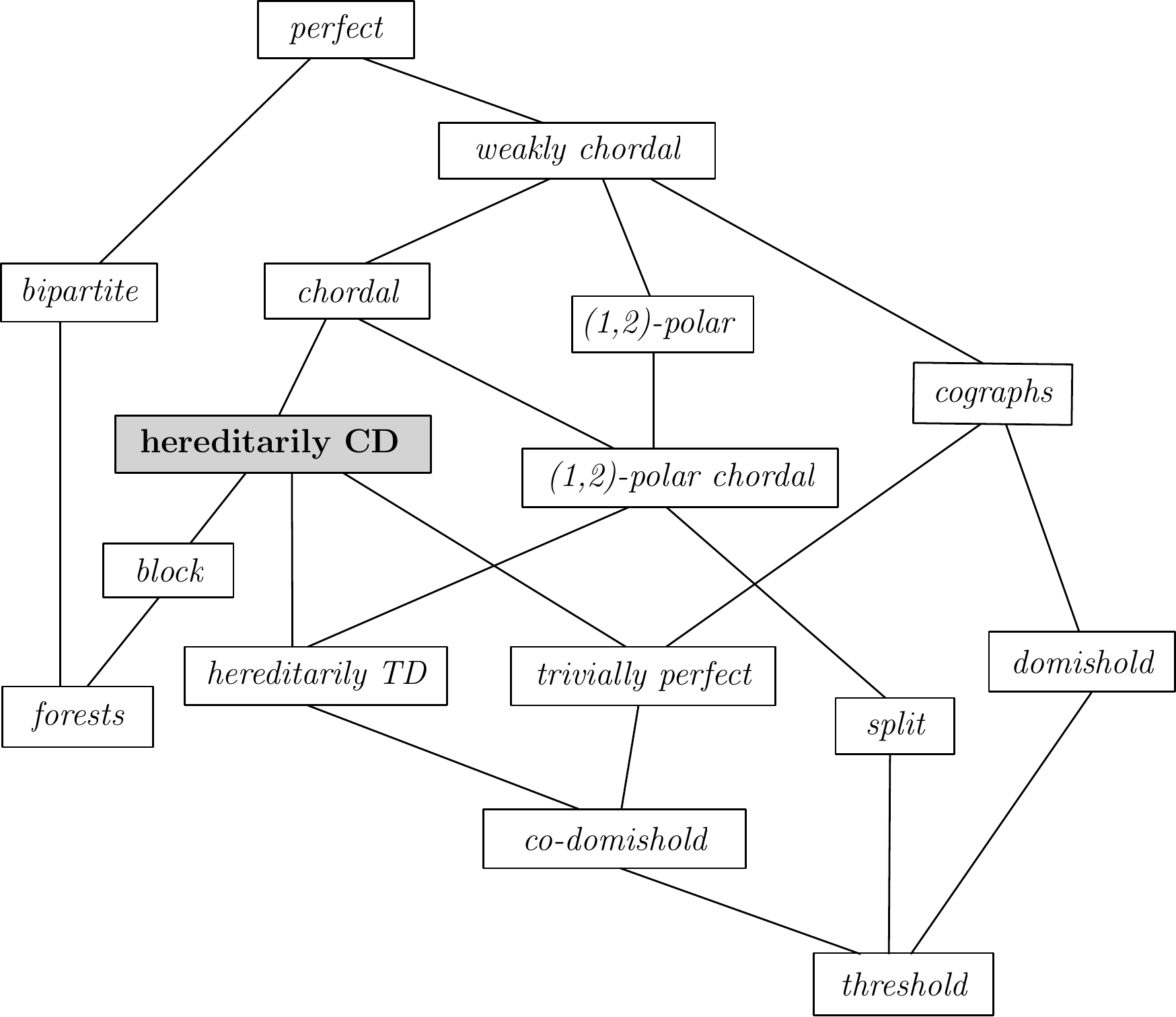}
  \end{center}
  \caption{A Hasse diagram depicting the inclusion relations within several families of perfect graphs, focused around the class of hereditarily connected-domishold graphs.}
\label{fig:Hasse}
\end{figure}

\section{Algorithmic aspects via vertex separators}\label{sec:algo}

In this section, we build on the above results, together with some known results from the literature on connected dominating sets and minimal vertex separators in graphs, to study certain algorithmic aspects of the class of connected-domishold graphs and its hereditary variant.

\subsection{The recognition problems}\label{sec:recognition}

We start with computational complexity aspects of the problems of recognizing whether a given graph
is CD, resp.~hereditarily CD. For general graphs, the computational complexity of recognizing connected-domishold graphs is not known.
However, we will now show that the hypergraph approach outlined in Section~\ref{sec:hypergraphs} leads to a sufficient condition
for the problem to be polynomially solvable, capturing a large number of graph classes. The condition is expressed using the notion of minimal vertex separators. Recall that a \emph{$u,v$-separator} (for a pair of non-adjacent vertices $u$, $v$) is a set $S\subseteq V(G)\setminus \{u,v\}$ such that $u$ and $v$ are in different components of $G-S$ and that a $u,v$-separator is minimal if it does not contain any other $u,v$-separator. Recall also that a \emph{minimal vertex separator} in $G$ is a minimal $u,v$-separator for some non-adjacent vertex pair $u$, $v$.

A sufficient condition for the polynomial time solvability of the recognition problem for CD graphs in a class of graphs $\mathcal{G}$ is that there exists a polynomial $p$ such that every connected graph \hbox{$G\in \mathcal{G}$} has at most $p(|V(G)|)$ minimal vertex separators. This is the case for chordal graphs, which have at most $|V(G)|$ minimal vertex separators~\cite{MR0408312}, as well as for many other classes of graphs, including permutation graphs, circle graphs, circular-arc graphs, chordal bipartite graphs, trapezoid graphs, cocomparability graphs of bounded dimension, distance-hereditary graphs, and weakly chordal graphs~(see, e.g.,~\cite{MR1857397,MR1284728,MonTod16}). For a polynomial $p$, let $\mathcal{G}_p$ be the class of graphs with at most $p(|V(G)|)$ minimal vertex separators. Since every minimal cutset is a minimal vertex separator, every connected graph $G\in \mathcal{G}_p$ has at most $p(|V(G)|)$ minimal cutsets.

It is known that the set of all minimal vertex separators of a given connected $n$-vertex graph can be enumerated in output-polynomial time. More precisely, Berry et al.~\cite{MR1792122} have developed an algorithm solving this problem in time $\mathcal{O}(n^3|\Sigma|)$ where $\Sigma$ is the set of all minimal vertex separators of $G$, improving on earlier (independently achieved) running times of $\mathcal{O}(n^5|\Sigma|)$ due to Shen and Liang~\cite{MR1453864} and Kloks and Kratsch~\cite{MR1612600}. Based on these results, we derive the following.

\begin{theorem}\label{thm:poly-cutsets}
For every polynomial $p$ there is a polynomial time algorithm to determine whether a given connected graph $G\in \mathcal{G}_p$ is connected-domishold. In case of a yes instance, the algorithm also computes an integral CD structure of $G$.
\end{theorem}

\begin{proof}
Let $G = (V,E)\in \mathcal{G}_p$ be a connected graph that is the input to the algorithm.

The algorithm proceeds as follows. If $G$ is complete, then $G$ is connected-domishold and an integral CD structure of $G$ is returned, say $(w,t)$ with $w(x) = 1$ for all $x\in V(G)$ and $t = 1$.
Assume now that $G$ is not complete. First, using the algorithm of Berry et al.~\cite{MR1792122}, we compute in time $\mathcal{O}(|V(G)|^3p(|V(G)|))$ the set $\Sigma$ of all minimal vertex separators of $G$. Next, the cutset hypergraph, $\MCH(G)$, is computed by comparing each pair of sets in $\Sigma$ and discarding the non-minimal ones. Since $\MCH(G)$ is Sperner, there is a bijective correspondence between the hyperedges of $\MCH(G)$ and the prime implicants of the cutset function $f_G^{\textit{cut}}$; this yields the complete DNF of $f_G^{\textit{cut}}$. Finally, we run the algorithm given by Theorem~\ref{thm:thr-Bf} on the complete DNF of $f_G^{\textit{cut}}$.
If $f_G^{\textit{cut}}$ is not threshold, then we conclude that $G$ is not connected-domishold.
Otherwise, the algorithm returned an integral separating structure, say $(w,t)$, of $f_G^{\textit{cut}}$. In this case we return
$(w,w(V)-t)$ as a CD structure of $G$.

It is clear that the algorithm runs in polynomial time. Its correctness follows from Proposition~\ref{prop:c-dom-graphs}.
\end{proof}

\begin{sloppypar}
Let $\tilde{\mathcal{G}}$ be the largest hereditary graph class such that a connected graph $G\in \tilde{\mathcal{G}}$ is connected-domishold if and only if it is total domishold. By Theorem~\ref{thm:split}, class $\tilde{\mathcal{G}}$ is a generalization of the class of split graphs. Since there is a polynomial time algorithm for recognizing total domishold graphs~\cite{ChiMil13,MR3281177}, there is a polynomial time algorithm to determine whether a given connected graph $G\in \tilde{\mathcal{G}}$ is connected-domishold. This motivates the following question (which we leave open).
\end{sloppypar}

\begin{question}
What is the largest hereditary graph class $\tilde{\mathcal{G}}$ such that a connected graph $G\in \tilde{\mathcal{G}}$ is connected-domishold if and only if it is total domishold?
\end{question}

A polynomial time recognition algorithm for the class of hereditarily CD graphs can be derived from the characterization of hereditarily CD graphs in terms of forbidden induced subgraphs given by Theorem~\ref{thm:characterizations}.

\begin{proposition}
There exists a polynomial time algorithm to determine whether a given graph $G$ is hereditarily CD.
In the case of a yes instance, an integral CD structure of $G$ can be computed in polynomial time.
\end{proposition}

\begin{sloppypar}
\begin{proof}
One can verify in linear time that $G$ is chordal~\cite{MR2063679} and verifying that $G$ is also $\{F_{1}, F_{2}, H_1, H_2\}$-free can be done in time $\mathcal{O}(|V(G)|^8)$. Therefore, we only have to show that we can check in polynomial time that $G$ does not contain an induced subgraph of the form $H_i$ for each $i>2$.
Observe that for all $i>2$ the graph $H_i$ contains an induced subgraph isomorphic to $2D$, the union of two diamonds (see Fig.~\ref{fig:forbidden-induced-subgraphs} and Fig.~\ref{fig:kite}).
In $\mathcal{O}(|V(G)^8|)$ time, we can enumerate all induced subgraphs $F$ of $G$ isomorphic to $2D$. For each such subgraph $F$ we have to verify whether it can be extended to an induced subgraph of the form $H_i$, for some $i>2$. We do this as follows. Let $D_1$ and $D_2$ be the connected components (diamonds) of $F$. Furthermore, let $u_1, u_2$ be the two vertices of degree $2$ in $D_1$ and similarly let $v_1,v_2$ be the two vertices of degree $2$ in $D_2$. Now we can verify that $F$ is not contained in any induced subgraph of $G$ isomorphic to $H_i$ (for some $i>2$) by checking for each pair $u_i,v_j$, with $i,j\in \{1,2\}$, that $u_i$ and $v_j$ belong to different components of $G-(N_{G-u_i}[V(D_1)\setminus\{u_i\}]\cup N_{G-v_j}[V(D_2)\setminus\{v_j\}]).$ This can be done in polynomial time and consequently the recognition of hereditarily CD graphs is a polynomially solvable problem.

The second part of the theorem follows from Theorem~\ref{thm:poly-cutsets}, since every hereditarily CD graph is chordal and chordal graphs
are a subclass of $\mathcal{G}_p$ for the polynomial $p(n) = n$~\cite{MR0408312}.
\end{proof}
\end{sloppypar}

It might seem conceivable that a similar approach as the one used in Theorem~\ref{thm:poly-cutsets} could be used to develop an efficient
algorithm for recognizing connected-domishold graphs in classes of graphs with only polynomially many minimal connected dominating sets. However,
it is not known whether there exists an output-polynomial time algorithm for the problem of enumerating minimal connected dominating sets.
In fact, as shown by Kant\'e et al.~\cite{KLMN}, even when restricted to split graphs, this problem is equivalent to the well-known {\sc Trans-Enum} problem in hypergraphs, the problem of enumerating the inclusion-minimal transversals of a given hypergraph. The {\sc Trans-Enum} problem has been intensively studied but it is still open whether there exists an output-polynomial time algorithm for the problem (see, e.g., the survey~\cite{MR2437000}).

\subsection{The weighted connected dominating set problem}\label{sec:WCDS}

\begin{sloppypar}
The {\sc Weighted Connected Dominating Set (WCDS)} problem takes as input a connected graph $G$ together with
a cost function $c:V(G)\to \mathbb{R}^+$, and the task is to compute a connected dominating set of minimum total cost, where the cost of a set
$S\subseteq V(G)$ is defined, as usual, as $c(S) = \sum_{v\in S}c(v)$. The WCDS problem has been studied extensively due to its many applications in networking (see, e.g., ~\cite{MR2114596,MR2986095,Wu2001}). The problem is NP-hard not only for general graphs~\cite{MR1605684} but also for split graphs~\cite{LP83}, chordal bipartite graphs~\cite{MR918093}, circle graphs~\cite{MR1206328}, and cocomparability graphs~\cite{MR1491477}. Polynomial time algorithms for the problem were developed for interval graphs~\cite{MR1622646} and more generally for trapezoid graphs~\cite{MR2321555} and circular-arc graphs~\cite{MR1622646,MR2118302}, as well as for distance-hereditary graphs~\cite{MR1650455}.

In this section, we will identify further graph classes where the WCDS problem is polynomially solvable, including the class of $F_2$-free split graphs (see Fig.~\ref{fig:F2}). This result is interesting in view of the fact that for split graphs, the WCDS problem is not only \NP-hard but also hard to approximate, even in the unweighted case. This can be seen as follows: Let $\mathcal{H} = (V,E)$ be a Sperner hypergraph with $\emptyset,V\notin E$ and let $G$ be its split-incidence graph. Then $G$ is a connected split graph without universal vertices, hence $\MCH(G) = \MNH(G)$ by Lemma~\ref{lem:split}.
It can be seen that the hyperedge set of $\MNH(G)$ is exactly $E$, and therefore Proposition~\ref{prop:c-dom-ms} implies that
the problem of finding a minimum connected dominating set in $G$ is equivalent to the {\sc Hitting Set} problem in hypergraphs, the problem of
finding a minimum transversal of a given hypergraph. This latter problem is known to be equivalent to the well-known {\sc Set Cover} problem and
hence inapproximable in polynomial time to within a factor of $(1-\epsilon)\log |V|$, for any $\epsilon>0$, unless {\sf P = NP}~\cite{MR3238990}. It follows that the WCDS problem is
hard to approximate to within a factor of $(1-\epsilon)\log |V(G)|$ in the class of split graphs.
\end{sloppypar}

We will show that the WCDS problem is polynomially solvable in the class of hereditarily CD graphs; the result for $F_2$-free split graphs will then follow. Our approach is based on connections with vertex separators and Boolean functions. First, we recall the following known results about: (i) the relation between the numbers of prime implicants of a threshold Boolean function and its dual, and (ii) the complexity of dualizing threshold Boolean functions. These results were proved in the more general
context of regular Boolean functions (as well as for other generalizations, see, e.g.,~\cite{BorosRRR}).

\begin{theorem}\label{thm:dual}
Let $f$ be an $n$-variable threshold Boolean function having exactly $q$ prime implicants.
Then:
\begin{enumerate}
  \item (Bertolazzi and Sassano~\cite{BS87}, Crama~\cite{CRAMA198779}, see also~\cite[Theorem 8.29]{MR2742439}) The dual function $f^d$ has at most $N$ prime implicants, where $N$ is the total number of variables in the complete DNF of~$f$.
  \item (Crama and Hammer~\cite[Theorem 8.28]{MR2742439} and Peled and Simeone~\cite{MR1272494}) There is an algorithm running in time $\mathcal{O}(n^2q)$ that, given the complete DNF of $f$, computes the complete DNF of the dual function $f^d$.
\end{enumerate}
\end{theorem}

The algorithm by Crama and Hammer~\cite{MR2742439} is already presented as having time complexity $\mathcal{O}(n^2q)$, while the
one by Peled and Simeone~\cite{MR1272494} is claimed to run in time $\mathcal{O}(nq)$. However, since $f^d$ can have $\mathcal{O}(nq)$ prime implicants, the total size of the output is of the order $\mathcal{O}(n^2q)$. The time complexity $\mathcal{O}(nq)$ of the algorithm by Peled and Simeone relies on the assumption that the algorithm outputs the prime implicants of the dual function one by one, each time overwriting the previous prime implicant (with a constant number of operations per implicant on average).

The relation between the numbers of prime implicants of a threshold Boolean function and its dual given by Theorem~\ref{thm:dual} implies that classes of connected-domishold graphs with only polynomially many minimal cutsets are exactly the same as the classes of connected-domishold graphs with only polynomially many minimal connected dominating sets. More precisely:

\begin{lemma}\label{lem:duality}
Let $G=(V,E)$ be an $n$-vertex connected-domishold graph that is not complete. Let $\nu_c$ (resp.~$\nu_s$) denote the number of minimal connected dominating sets (resp.~of minimal cutsets) of $G$. Then $\nu_s\le (n-2)\nu_c$ and $\nu_c\le (n-2)\nu_s$.
\end{lemma}

\begin{proof}
By Proposition~\ref{prop:c-dom-graphs}, the cutset function $f_G^{\textit{cut}}$ is threshold.
Function $f_G^{\textit{cut}}$ is an $n$-variable function with exactly $\nu_s$ prime implicants in its complete DNF.
Recall from the proof of Proposition~\ref{prop:c-dom-graphs} that the dual function $(f_G^{\textit{cut}})^d$
takes value $1$ precisely on the vectors $x\in \{0,1\}^V$ whose support is a connected dominating set of $G$.
Therefore, the prime implicants of $(f_G^{\textit{cut}})^d$ are in bijective correspondence with the minimal connected dominating sets of $G$
and the number of prime implicants of $(f_G^{\textit{cut}})^d$ is exactly $\nu_c$.
Since every minimal cutset of $G$ has at most $n-2$ vertices, Theorem~\ref{thm:dual} implies that $\nu_c\le (n-2)\nu_s$, as claimed.

Conversely, since $f_G^{\textit{cut}} = ((f_G^{\textit{cut}})^d)^d$, the inequality $\nu_s\le (n-2)\nu_c$ can be proved by a similar approach, provided we show that every minimal connected dominating set of $G$ has at most $n-2$ vertices. But this is true since if $D$ is a connected dominating set of $G$ with at least $n-1$ vertices, with $V(G)\setminus\{u\}\subseteq D$ for some $u\in V(G)$, then a smaller connected dominating set $D'$ of $G$ could be obtained by fixing an arbitrary spanning tree $T$ of $G[D]$ and deleting from $D$ an arbitrary leaf $v$ of $T$ such that $N_G(u)\neq \{v\}$. (Note that since $G$ is connected but not complete, it has at least three vertices, hence $T$ has at least two leaves.) This completes the proof.
\end{proof}

We now have everything ready to derive the main result of this section. Recall that for a polynomial $p$, we denote by $\mathcal{ G}_p$ the class of graphs with at most $p(|V(G)|)$ minimal vertex separators.

\begin{theorem}\label{thm:poly}
For every nonzero polynomial $p$, the set of minimal connected dominating sets of an $n$-vertex
connected-domishold graph from $\mathcal{G}_{p}$ has size at most ${\mathcal{O}}(n\cdot {p}(n))$ and
can be computed in time ${\mathcal{O}}(n\cdot{p}(n)\cdot(n^2+{p}(n)))$.
In particular, the WCDS problem is solvable in polynomial time in the class of connected-domishold graphs from $\mathcal{G}_{p}$.
\end{theorem}

\begin{proof}
Let $p$ and $G$ be as in the statement of the theorem and let $\mathcal{CD}(G)$ be the set of minimal connected dominating sets of $G$.
If $G$ is complete, then $\mathcal{CD}(G) =\{\{v\}: v\in V(G)\}$ and thus $|\mathcal{CD}(G)| = n = {\mathcal{O}}(n\cdot{p}(n))$ (since the polynomial is nonzero). Otherwise, we can apply Lemma~\ref{lem:duality} to derive $|\mathcal{CD}(G)|\le (n-2)\cdot {p}(n)$.

A polynomial time algorithm to solve the WCDS problem for a given connected-domishold graph $G\in \mathcal{G}_{p}$ with respect to a cost function $c:V(G)\to \mathbb{R}^+$ can be obtained as follows. First, we may assume that $G$ is not complete, since otherwise we can return a set $\{v\}$ where $v$ is a vertex minimizing $c(v)$. We use a similar approach as in the proof of Theorem~\ref{thm:poly-cutsets}.
Using the algorithm of Berry et al.~\cite{MR1792122}, we compute in time $\mathcal{O}(n^3p(n))$ the set $\Sigma$ of all minimal vertex separators of $G$. We can assume that each minimal vertex separator has its elements listed according to some fixed order of $V(G)$ (otherwise, we can sort them in time $\mathcal{O}(n\cdot p(n))$ using, e.g., bucket sort). The cutset hypergraph, $\MCH(G)$, is then computed by comparing each pair of sets in $\Sigma$ and discarding the non-minimal ones; this can be done in time $\mathcal{O}(n\cdot (p(n))^2)$. The cutset hypergraph directly corresponds to the complete DNF of the cutset function $f_G^{\textit{cut}}$.

The next step is to compute the complete DNF of the dual function $(f_G^{\textit{cut}})^d$. By Theorem~\ref{thm:dual}, this can be done in time
${\mathcal{O}}(n^2\cdot p(n))$. Since each term of the DNF is a prime implicant of $(f_G^{\textit{cut}})^d$ and the prime implicants of $(f_G^{\textit{cut}})^d$ are in bijective correspondence with the minimal connected dominating sets of $G$, we can read off from the DNF all the minimal
connected dominating sets of $G$. The claimed time complexity follows.

Once the list of all minimal connected dominating sets is available, a polynomial time algorithm for the WCDS problem on $(G,c)$ follows immediately.
\end{proof}

In the case of chordal graphs, we can improve the running time by using one of the known linear-time algorithms for listing the minimal vertex separators of a given chordal graph due to Kumar and Veni Madhavan~\cite{KumarM98}, Chandran and Grandoni~\cite{MR2204112}, and Berry and Pogorelcnik~\cite{MR2816655}.

\begin{sloppypar}
\begin{theorem}\label{thm:chordal}
Every $n$-vertex connected-domishold chordal graph has at most $\mathcal{O}(n^2)$ minimal connected dominating sets, which can be enumerated in time $\mathcal{O}(n^3)$. In particular, the WCDS problem is solvable in time $\mathcal{O}(n^3)$ in the class of connected-domishold chordal graphs.
\end{theorem}
\end{sloppypar}

\begin{proof}
Let $G$ be an $n$-vertex connected-domishold chordal graph. The theorem clearly holds for complete graphs, so we may assume that $G$ is not complete.
Since $G$ is chordal, it has at most $n$ minimal vertex separators~\cite{MR0408312}; consequently, $G$ has at most $n$ minimal cutsets.
Since $G$ is connected-domishold, it has at most $n(n-2)$ minimal connected dominating sets, by Lemma~\ref{lem:duality}.

The minimal connected dominating sets of $G$ can be enumerated as follows. First, we compute all the $\mathcal{O}(n)$ minimal vertex separators of $G$ in time $\mathcal{O}(n+m)$ (where $m = |E(G)|$) using one of the known algorithms for this problem on chordal graphs~\cite{MR2816655,MR2204112,KumarM98}.
Assuming again that each minimal vertex separator has its elements listed according to some fixed order of $V(G)$, we then
eliminate those that are not minimal cutsets in time $\mathcal{O}(n^3)$, by directly comparing each of the $\mathcal{O}(n^2)$ pairs for inclusion.

The list of $\mathcal{O}(n)$ minimal cutsets of $G$ yields its cutset function, $f_G^\emph{ms}$.
The list of minimal connected dominating sets of $G$ can be obtained in time $\mathcal{O}(n^3)$
by dualizing $f_G^\emph{ms}$ using one of the algorithms given by Theorem~\ref{thm:dual}.
The WCDS problem can now be solved in time $\mathcal{O}(n^3)$ by evaluating the cost of each of the
$\mathcal{O}(n^2)$ minimal connected dominating sets and outputting one of minimum cost.
\end{proof}

From Theorem~\ref{thm:chordal} we derive two new polynomially solvable cases of the WCDS problem.
Recall that the graphs $F_1$, $F_2$, and a general member of the family $\{H_i\}$ are depicted in
Fig.~\ref{fig:forbidden-induced-subgraphs}.

\begin{corollary}
The WCDS problem is solvable in time $\mathcal{O}(n^3)$ in the class of
$\{F_1,F_2,H_1,H_2,\ldots\}$-free chordal graphs and in particular in the class of $F_2$-free split graphs.
\end{corollary}

\begin{sloppypar}
\begin{proof}
By Theorem~\ref{thm:characterizations}, every $\{F_1,F_2,H_1,H_2,\ldots\}$-free chordal graphs is (hereditarily) CD so Theorem~\ref{thm:chordal} applies. The statement for $F_2$-free split graphs follows from Corollary~\ref{cor:split}.
\end{proof}
\end{sloppypar}

We conclude this section with two remarks, one related to Theorem~\ref{thm:chordal} and one related to Theorems~\ref{thm:poly-cutsets} and~\ref{thm:poly}.

\begin{sloppypar}
\begin{remark}
The bound $\mathcal{O}(n^2)$ given by Theorem~\ref{thm:chordal} on the number of minimal connected dominating sets in an $n$-vertex connected-domishold chordal graph is sharp. There exist $n$-vertex connected-domishold chordal graphs with $\Theta(n^2)$ minimal connected dominating sets. For instance, let $S_n$ be the split graph with $V(S_n) = K\cup I$ where $K = \{u_1,\ldots, u_n\}$ is a clique, $I = \{v_1,\ldots, v_n\}$ is an independent set, $K\cap I = \emptyset$, and for each $i\in [n]$, vertex $u_i$ is adjacent to all vertices of $I$ except $v_i$. Since every vertex in $I$ has a unique non-neighbor in $K$, we infer that
$S_n$ is $F_2$-free. Therefore, by Corollary~\ref{cor:split} graph $S_n$ is a (hereditarily) connected-domishold graph.
Note that every set of the form $\{u_i,u_j\}$ where $1\le i<j\le n$ is a minimal connected dominating set of $S_n$.
It follows that $S_n$ has at least ${n\choose 2} = \Theta(|V(S_n)|^2)$ minimal connected dominating sets.
\end{remark}
\end{sloppypar}

\begin{remark}
Theorems~\ref{thm:poly-cutsets} and~\ref{thm:poly} motivate the question of whether there is a polynomial $p$ such that every connected CD graph $G$ has at most $p(|V(G)|)$ minimal vertex separators.
As shown by the following family of graphs, this is not the case.
For $n\ge 2$, let $G_n$ be the graph obtained from the disjoint union of $n$ copies of the $P_4$,
say $(x_i,a_i,b_i,y_i)$ for $i = 1,\ldots, n$, by identifying all vertices $x_i$ into a single vertex $x$,
all vertices $y_i$ into a single vertex $y$, and for each vertex $z$ other than $x$ or $y$,
adding a new vertex $z'$ and making it adjacent only to $z$. It is not difficult to see that
$G_n$ has exactly two minimal CD sets, namely $\{a_1,\ldots, a_n\}\cup \{b_1,\ldots, b_n\}\cup \{v\}$ for $v\in \{x,y\}$.
A CD structure of $G_n$ is given by $(w,t)$ where $t = 4n+1$, $w(x) = w(y) = 1$,
$w(a_i) = w(b_i) = 2$ for all $i\in \{1,\ldots, n\}$ and $w(z) = 0$ for all other vertices $z$.
Therefore, $G_n$ is CD. However, $G_n$ has $4n+2$ vertices and $2^n$ minimal $x,y$-separators, namely all sets of the form
$\{c_1,\ldots, c_n\}$ where $c_i\in \{a_i,b_i\}$ for all $i$.
\end{remark}

\section{Proof of Lemma~\ref{diamondLemma} (Diamond Lemma)}\label{sec:proof}

In the proof of the Diamond Lemma, we use the following notation.
We write $u\sim v$ (resp.~$u\nsim v$) to denote the fact that two vertices $u$ and $v$ are adjacent (resp.~non-adjacent). Given two vertex sets $A$ and $B$ in a graph $G$, we denote by $e(A,B)$ the number of edges with one endpoint in $A$ and one endpoint in $B$. A {\em pattern} is a triple $(V,E,F)$ where $G=(V,E)$ is a graph and $F$ is a subset of non-adjacent vertex pairs of $G$. We say that a graph $G'$ {\em realizes a pattern} $(V,E,F)$ if $V(G')=V$ and $E\subseteq E(G')\subseteq E\cup F$.

\begin{figure}[h]
  \centering
  \includegraphics[width=0.8\linewidth]{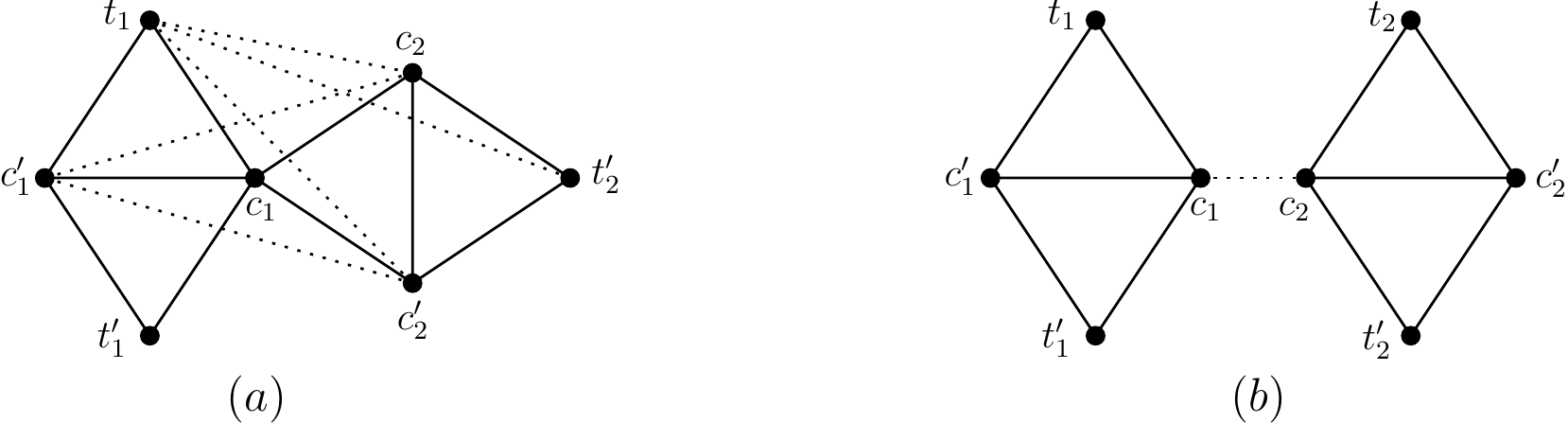}\\
  \caption{Two patterns $(V,E,F)$ used in the proofs. Graphs $(V,E)$ are depicted with solid lines. Possible additional edges (elements of $F$) are depicted with dotted lines. }\label{Fig:Lemmas}
\end{figure}

We start with a lemma.

\begin{lemma}\label{lem:c_1-not-adjacent-to-C_2}
Let $G$ be a connected chordal graph and let $H$ be an induced subgraph of $G$ that realizes the pattern in Fig.~\ref{Fig:Lemmas}\,$(a)$. Moreover, suppose that:
\begin{enumerate}[1)]
  \item vertices $t_1$ and $t_1'$ are in different components of $G-\{c_1,c_1'\}$,and \label{lem-ppt1}
  \item the component of $G-\{c_1,c_1'\}$ containing $\{c_2,c_2',t_2'\}$ has a vertex dominating $\{c_1,c_1'\}$. \label{lem-ppt2}
\end{enumerate}
Then $G$ contains $F_1$ or $F_2$ as an induced subgraph.
\end{lemma}

\begin{proof}
By contradiction. Suppose that $G$ and $H$ satisfy the assumptions of the lemma, but $G$ is $\{F_1,F_2\}$-free. We will first show that none of the dotted edges can be present in $H$.
We infer that $c_1'\nsim c_2$ and $c_1'\nsim c_2'$, or an induced $F_1$ or $F_2$ arises on the vertex set $V(H)\setminus\{t_1\}$, depending on whether one or both edges are present.
Next, $t_1\nsim t_2'$, since otherwise a $4$-cycle arises on the vertex set $\{t_1,c_1,c_2',t_2'\}$ (if $t_1\nsim c_2'$) or an induced $F_1$ arises on the vertex set $V(H)\setminus\{c_2\}$ (otherwise).
Finally, we infer that $t_1\nsim c_2$ and $t_1\nsim c_2'$, or otherwise an induced $F_1$ or $F_2$ arises on the vertex set $V(H)\setminus\{t_1'\}$, depending whether one or both edges are present.

Let $K$ be the component of $G-\{c_1,c_1'\}$ such that $V_2'=\{c_2,c_2',t_2'\}\subseteq V(K)$, and let $w\in V(K)$ be a vertex dominating $\{c_1,c_1'\}$ that is closest to $V_2'$ in $K$. Clearly, $w\notin V_2'$.
We will now show that $w\nsim v$ for any $v\in V_2'$. Suppose for a contradiction that $w\sim v$ for some $v\in V_2'$.
Note that $w\notin \{t_1,t_1'\}$ since there are no edges between the sets $\{t_1,t_1'\}$ and $V_2'$.
Furthermore, property $(\ref{lem-ppt1})$ implies that there exists some $t\in \{t_1,t_1'\}$ such that $w\nsim t$.
Suppose that $w\sim t_2'$. Then $w\sim c_2$, since otherwise a $4$-cycle arises on the vertex set $\{w,c_1,c_2,t_2'\}$.
But now the vertex set $\{t_2',c_2,w,c_1,c_1',t\}$ induces a copy of $F_1$ in $G$. Therefore $w\nsim t_2'$, and
an induced $F_1$ or $F_2$ arises on the vertex set $V_2'\cup \{w,c_1,c_1'\}$, depending on whether $w$ is adjacent to one or both vertices in $\{c_2,c_2'\}$. This contradiction shows that $w$ has no neighbor in $V_2'$.

Let $P=(w=w_1,\dots,w_k)$ with $w_k\in V_2'$ be a shortest $w,V_2'$-path in $K$. Note that $k\geq 3$ and the choice of $P$ implies that for all $i\in \{1,\dots,k-2\}$  vertex $w_i$ is not adjacent to any vertex in $V_2'$. In order to avoid an induced cycle of length at least $4$ within $V(P)\cup V_2'\cup \{c_1\}$, we infer that vertex $c_1$ must be adjacent to all the internal vertices of $P$ (that is, to $w_2,\dots,w_{k-1}$).
Next we infer that $w_{k-1}\sim t_2'$, since otherwise the vertex set $V_2'\cup \{c_1,w_{k-1},w_{k-2}\}$ induces a copy of $F_1$ or $F_2$ (depending on the number of edges between $w_{k-1}$ and $\{c_2,c_2'\}$).
Moreover, to avoid an induced $4$-cycle on the vertex set $\{t_2',w_{k-1},c_1,c_2\}$, we infer that $w_{k-1}\sim c_2$.
But now an induced $F_1$ arises on the vertex set $\{t_2',c_2,c_1,w_{k-1},w_{k-2},w_{k-3}\}$ (where if $k=3$, we set $w_0=c_1'$).
This last contradiction completes the proof of Lemma~\ref{lem:c_1-not-adjacent-to-C_2}.
\end{proof}

\begin{diamondLemma}[Diamond Lemma (restated)]
Let $G$ be a connected chordal graph. Suppose that $G$ contains two induced diamonds $D_1=(V_1,E_1)$ and $D_2=(V_2,E_2)$ such that:
\begin{enumerate}[(i)]
  \item $C_1\cap C_2=\emptyset$.\label{ppt1}
  \item If no vertex in $C_1$ is adjacent to a vertex of $C_2$, then every minimal $C_1,C_2$-separator in $G$ is of size one.\label{ppt4}
  \item For each $j\in\{1,2\}$ the tips (i.e., $t_j,t_j'$) of $D_j$ belong to different components of $G-C_j$.\label{ppt2}
  \item For $j\in\{1,2\}$ every component of $G-C_j$ has a vertex that dominates $C_j$.\label{ppt3}
\end{enumerate}
Then $G$ has an induced subgraph isomorphic to $F_1,F_2,$ or $H_i$ for some $i\geq 1$, where the graphs $F_1$, $F_2$, and a general member of the family $\{H_i\}$ are depicted in Fig.~\ref{fig:forbidden-induced-subgraphs}.
\end{diamondLemma}

\begin{cproof}
We will prove the Diamond Lemma by contradiction through a series of claims. Let $G$ be a connected chordal graph and let $D_1$ and $D_2$ be two induced diamonds with properties $(\ref{ppt1})-(\ref{ppt3})$ in $G$. Suppose for a contradiction that $G$ is $\{F_1,F_2,H_1,H_2,\dots\}$-free.

\begin{claim2}\label{claim:nonDominatingTip}
For each $j\in \{1,2\}$, there exists some $t\in T_j$ such that $N[t]\cap C_{3-j}=\emptyset$ (that is, each diamond has a tip that does not dominate any center of the other diamond).
\end{claim2}

\begin{proof}
Suppose that both tips of $D_j$ dominate $C_{3-j}$. Then $T_j$ belongs to one component of $G-C_j$, contradicting property $(\ref{ppt2})$.
\end{proof}

\begin{claim2}\label{claim:OneTip}
If there exists some $t\in T_1\cap T_2$, then $T_1\cap T_2 =\{t\}$ and $T_j\cap C_{3-j}=\emptyset$  for $j\in \{1,2\}$.
\end{claim2}

\begin{proof}
Follows immediately from Claim~\ref{claim:nonDominatingTip} and property $(\ref{ppt2})$.
\end{proof}

\begin{claim2}\label{claim:intersectionleq1}
$|V_1\cap V_2|\leq 1$.
\end{claim2}
\begin{proof}
First note that we have $|T_1\cap V_{2}|\leq 1$, since otherwise $T_1=T_2$, contradicting property $(\ref{ppt2})$. Observe also that by property $(\ref{ppt1})$ we have $C_1\cap V_{2}\subseteq C_1\cap T_{2}$, implying that $|C_1\cap V_2|\leq 1$. Consequently $|V_1\cap V_2|\leq 2$.


Now suppose for a contradiction that $|V_1\cap V_2|=2$. By property $(\ref{ppt1})$ and Claim~\ref{claim:OneTip} we may assume without loss of generality that $c_1=t_2$ and $t_1'=c_2'$.
To avoid an induced $4$-cycle on the set $T_1\cup T_2$ we infer that $t_1\nsim t_2'$. Furthermore, property $(\ref{ppt2})$ implies that $c_1'\nsim t_2'$ and $c_2\nsim t_1$. But now the set $V_1\cup V_2$ induces a copy of $F_1$ (if $c_1'\nsim c_2)$ or a copy of $F_2$ (otherwise).
\end{proof}

\begin{claim2}\label{claim:intersectionTips}
If $V_1\cap V_2=\{v\}$ then $v\in T_1\cap T_2$.
\end{claim2}

\begin{proof}
Suppose for a contradiction that $V_1\cap V_2=\{v\}$, and $v\notin T_1\cap T_2$. Property $(\ref{ppt1})$ implies that $v\in T_j\cap C_{3-j}$ for some $j\in\{1,2\}$, say $v=c_1=t_2$. Claim~\ref{claim:nonDominatingTip} implies (without loss of generality) that $t_1'\nsim c_2$ and $t_1'\nsim c_2'$.
Property $(\ref{ppt2})$ implies that $c_1'\nsim t_2'$. Note that $t_1'\nsim t_2'$, or otherwise a $4$-cycle arises on the vertex set $\{t_1',c_1,c_2,t_2'\}$.
Now the subgraph of $G$ induced by $V_1\cup V_2$ realizes the pattern depicted in Fig.~\ref{Fig:Lemmas}\,$(a)$ and we apply Lemma~\ref{lem:c_1-not-adjacent-to-C_2} to derive a contradiction.
\end{proof}

\begin{claim2}\label{claim:intersectionEmpty}
$V_1\cap V_2=\emptyset$.
\end{claim2}
\begin{proof}
Suppose for a contradiction that $V_1\cap V_2\neq \emptyset$. Claim~\ref{claim:intersectionleq1} implies that $V_1\cap V_2=\{v\}$ and by Claim~\ref{claim:intersectionTips}, $v\in T_1\cap T_2$. Without loss of generality we may assume that $t_1=t_2$.
Claim~\ref{claim:nonDominatingTip} implies that there is no edge between $t_1'$ and $C_2$ and between $t_2'$ and $C_1$. Furthermore, we must have $t_1'\nsim t_2'$ since otherwise $G$ contains an induced $4$-cycle on the vertex set $\{t_1',c_1,c_2,t_2'\}$ (if $c_1\sim c_2$) or an induced $5$-cycle on the vertex set $\{t_1',c_1,t_1,c_2,t_2'\}$ (otherwise).

It remains to analyze the edges between $C_1$ and $C_2$. Clearly, $e(C_1,C_2)\in\{0,1,\dots,4\}$. Notice that
$$e(C_1,C_2) = \left\{
             \begin{array}{ll}
             0 & \hbox{implies an induced $H_1$ on the set $V_1\cup V_2$};\\
             1 & \hbox{implies an induced $F_1$ on the vertex set $(V_1\cup V_2)\setminus \{t_1'\}$};\\
             3 & \hbox{implies an induced $F_1$ on the vertex set $(V_1\cup V_2)\setminus \{t_1\}$};\\
             4 & \hbox{implies an induced $F_2$ on the vertex set $(V_1\cup V_2)\setminus \{t_1\}$}.
             \end{array}
           \right.$$
Consequently $e(C_1,C_2)=2$, and without loss of generality, to avoid an induced $4$-cycle, we may assume that $c_1\sim c_2$ and $c_1\sim c_2'$. But now an induced $F_2$ arises on the vertex set $(V_1\cup V_2)\setminus \{t_1'\}$.
\end{proof}

In the rest of the proof of the Diamond Lemma we consider the edges between $V_1$ and $V_2$.
By Claim~\ref{claim:nonDominatingTip} and property $(\ref{ppt2})$
we may assume without loss of generality the following.

\begin{assumption}\label{ass1}
$e(\{t_1'\},V_2) = e(\{t_2'\},V_1) = 0$.
\end{assumption}

Therefore, it remains to consider only the (non)edges between $\{t_1\}$ and $C_2$, between $\{t_2\}$ and $C_1$, between $C_1$ and $C_2$, and between $\{t_1\}$ and $\{t_2\}$.

\begin{claim2} \label{claim:Centers01}
$e(C_1,C_2)\leq 1$.
\end{claim2}

\begin{proof}
Clearly, $e(C_1,C_2)\le 4$.
Note that if $e(C_1,C_2)\in \{3,4\}$, then the vertex set $(V_1\cup V_2)\setminus \{t_1,t_2\}$ induces either a copy of $F_1$ or a copy of $F_2$.
Furthermore, if $e(C_1,C_2)=2$, then, to avoid an induced $4$-cycle, we may assume without loss of generality that $c_1\sim c_2$ and $c_1\sim c_2'$.
Now the subgraph of $G$ induced by $(V_1\cup V_2)\setminus\{t_2\}$ realizes the pattern depicted in Fig.~\ref{Fig:Lemmas}\,$(a)$ (without any additional edges) and we apply Lemma~\ref{lem:c_1-not-adjacent-to-C_2} to derive a contradiction.
\end{proof}

By Claim~\ref{claim:Centers01} we may assume without loss of generality
the following.

\begin{assumption}\label{ass2}
$c_1'\nsim c_2$, $c_1'\nsim c_2'$, and $c_1\nsim c_2'$.
\end{assumption}

\begin{claim2}\label{claim:TipsCenters01}
$e(t_j,C_{3-j})\leq 1$ for any $j\in \{1,2\}$; moreover, if $e(t_j,C_{3-j})=1$ then we may assume without loss of generality that $t_j\sim c_{3-j}$.
\end{claim2}

\begin{proof}
Suppose for a contradiction that $e(t_j,C_{3-j})=2$.
To avoid an induced $H_1$ on the vertex set $(V_1\cup V_2)\setminus\{t_{3-j}\}$, we must have an edge between $C_1$ and $C_2$.
By Claim~\ref{claim:Centers01} we may assume that $c_1\sim c_2$, but now an induced $F_1$ arises on the vertex set $V_j\cup C_{3-j}$.

Suppose now that $e(t_j, C_{3-j})=1$. To see that we may assume without loss of generality, that $t_j\sim c_{3-j}$, note that this can be achieved by swapping $c_{3-j}$ and $c_{3-j}'$ (if necessary) when $c_1\nsim c_2$, while if $c_1\sim c_2$, then $t_j\sim c_{3-j}$, since otherwise the vertex set $\{t_j,c_1,c_2,c_{3-j}'\}$ induces a $4$-cycle in $G$.
\end{proof}

\begin{claim2}\label{claim:noTedge}
$t_1\nsim t_2$.
\end{claim2}

\begin{proof}
Suppose for a contradiction that $t_1\sim t_2$.
First we will show that $c_1\sim t_2$ or $c_2\sim t_1$. Suppose for a contradiction that $c_1\nsim t_2$, and $c_2\nsim t_1$.
Then an induced $H_2$ arises on the set $V_1\cup V_2$ (if $c_1\nsim c_2$) or an induced $4$-cycle on the vertex set $\{c_1,t_1,t_2,c_2\}$ (otherwise).

Without loss of generality we may assume that $c_1\sim t_2$.
Then Claim~\ref{claim:TipsCenters01} implies $t_2\nsim c_1'$, and to avoid an induced $H_1$ on the vertex set $(V_1\cup V_2)\setminus \{t_1'\}$, we must have an edge between $t_1$ and $C_2$ or $c_1\sim c_2$.
If $t_1\sim c_2$ then $t_1\nsim c_2'$ by Claim~\ref{claim:TipsCenters01}. But now the vertex set $C_1\cup C_2\cup \{t_1,t_2\}$ induces a copy of $F_1$ or $F_2$ (depending whether $c_1\sim c_2$ or not). Consequently $t_1\nsim c_2$.
If $t_1\sim c_2'$ then the vertex set $V_1\cup \{t_2,c_2'\}$ induces a copy of $F_1$. Therefore the only edge we can have is $c_1\sim c_2$, but now an induced $F_1$ arises on the vertex set $C_1\cup C_2\cup \{t_1,t_2\}$.
\end{proof}

\begin{claim2} \label{claim:TipsCenters0}
$e(\{t_j\},C_{3-j})=0$ for $j\in \{1,2\}$.
\end{claim2}

\begin{proof}
It suffices to consider the case $j=2$. Suppose for a contradiction that $e(\{t_2\},C_1)=1$.
By Claim~\ref{claim:TipsCenters01} we may assume $t_2\sim c_1$ and consequently $t_2\nsim c_1'$. Moreover Claim~\ref{claim:noTedge} implies that $t_1\nsim t_2$. Recall that by Assumption~\ref{ass1} we have $t_2\nsim t_1'$. Furthermore $e(\{t_1\},C_2)=0$, since otherwise, if $t_1\sim c_2'$, then an induced $4$-cycle arises on the vertex set $\{t_1,c_2',c_1,t_2\}$, and if
$t_1\sim c_2$, then either the vertex set $\{t_1,c_2,c_1,t_2\}$ induces a $4$-cycle (if $c_1\nsim c_2$) or the vertex set $C_1\cup C_2\cup \{t_1,t_2\}$ induces an $F_1$ (otherwise).

Let $K$ be the component of $G-C_1$ such that $V_2\subseteq V(K)$. By property $(\ref{ppt3})$ there exists a vertex in $V(K)$ that dominates $C_1$. Let $w\in V(K)$ be a vertex that dominates $C_1$ and is closest to $V_2$ in $K$. Clearly, $w\notin V_2$. First we will show that $w\nsim v$ for any $v\in V_2$. Suppose for a contradiction that $w\sim v$ for some $v\in V_2$.
Note that $w\notin T_1$ since there are no edges between the sets $T_1$ and $V_2$.
Furthermore, property $(\ref{ppt2})$ implies that there exists some $t\in T_2$ such that $w\nsim t$.
If $w\sim c_2'$, then $w\sim t_2$, or otherwise a $4$-cycle arises on the vertex set $\{w,c_1,t_2,c_2'\}$.
But now the set $C_1\cup \{w,c_2',t_2,t\}$ induces copy of $F_1$ in $G$. Therefore $w\nsim c_2'$.
If $w\sim t_2$, and in addition $c_1\nsim c_2$, then either an induced $H_1$ arises on the vertex set $V_2\cup C_1\cup \{w\}$ (if $w\nsim c_2$)
 or an induced $F_1$ arises on the vertex set $C_1\cup C_2\cup \{w,t_2\}$ (otherwise). Therefore $c_1\sim c_2$. But now the vertex set $C_1\cup C_2\cup \{t_2,w\}$ induces a copy of either $F_1$ (if $w\nsim c_2)$ or $F_2$ (otherwise). Therefore $w\nsim t_2$.
Also, $w\nsim c_2$, or otherwise either the vertex set $\{c_1,t_2,c_2,w\}$ induces a $4$-cycle (if $c_1\nsim c_2$) or the vertex set $C_1\cup C_2\cup \{w,t_2\}$ induces a copy of $F_1$ in $G$ (otherwise).
Furthermore $w\nsim t_2'$, or otherwise either the vertex set $T_2\cup \{c_1,c_2',w\}$ induces a $5$-cycle (if $c_1\nsim c_2$) or the vertex set $\{c_1,c_2,t_2',w\}$ induces a $4$-cycle (otherwise).

Let $P=(w=w_1,\dots,w_k)$ with $w_k\in V_2$ be a shortest $w,V_2$-path in $K$. Note that $k\geq 3$ and that the choice of $P$ implies that for all
$i\in \{1,\dots,k-2\}$ vertex $w_i$ is not adjacent to any vertex in $V_2$.
In order to avoid an induced cycle of length at least $4$ within $V(P)\cup V_2\cup \{c_1\}$, we infer that vertex $c_1$ must be adjacent to all the internal vertices of $P$ (that is $w_2,\dots,w_{k-1}$).

We claim that $w_{k-1}\sim t_2$. Suppose for a contradiction that $w_{k-1}\nsim t_2$.
Then if $w_{k-1}\sim c_2'$ an induced $4$-cycle arises on the vertex set $\{c_1,t_2,c_2',w_{k-1}\}$. If $w_{k-1}\sim c_2$, to avoid an induced $4$-cycle on the vertex set $\{w_{k-1},c_1,t_2,c_2\}$, we must have $c_1\sim c_2$. But now an induced $F_1$ arises on the vertex set $C_2\cup \{t_2,c_1,w_{k-1},w_{k-2}\}$.
If $w_{k-1}\sim t_2'$ either an induced $5$-cycle arises on the vertex set $T_2\cup \{c_1,c_2,w_{k-1}\}$ (if $c_1\nsim c_2$) or an induced $4$-cycle arises on the vertex set $\{t_2',w_{k-1},c_1,c_2\}$ (otherwise).
Therefore $w_{k-1}\sim t_2$. But now either an induced $H_1$ arises on the vertex set $V_2\cup \{w_{k-1},w_{k-2},c_1\}$ (if $c_1\nsim c_2$) or an induced $F_1$ arises on the vertex set $V_2\cup \{w_{k-1},c_1\}$.
\end{proof}

Let $H$ be the subgraph of $G$ induced by $V_1\cup V_2$.
Note that $H$ realizes the pattern in Fig.~\ref{Fig:Lemmas}\,$(b)$.
Let $K^{-1}_2$ be the component of $G-C_1$ containing $V_2$ and let $U^{-1}$ be the set of vertices in $K^{-1}_2$ that dominate $C_1$.
By property $(\ref{ppt3})$, set $U^{-1}$ is non-empty.
Let $u^{-1}$ be a vertex in $U^{-1}$ that is closest in $K^{-1}_2$ to $C_2$.
Graph $K^{-2}_1$ and vertex $u^{-2}$ are defined similarly.

By property $(\ref{ppt2})$ we may assume without loss of generality the following.

\begin{assumption}\label{ass3}
$t_1' \notin V(K^{-1}_2)$ and $t_2' \notin V(K^{-2}_1)$.
Furthermore, $u^{-i}\nsim t_1'$ and $u^{-i}\nsim t_2'$ for $i\in \{1,2\}$.
\end{assumption}

\begin{claim2}\label{claim:TwoPaths}
Vertices $u^{-1}$ and $u^{-2}$ are distinct and non-adjacent, and
at least one of the sets
\hbox{$N(u^{-1})\cap V_2$}, $N(u^{-2})\cap V_1$ is empty.
\end{claim2}

\begin{proof}
First we prove that $u^{-1}\nsim c_2$ or $u^{-1}\nsim c_2'$. Suppose for a contradiction that $e(\{u^{-1}\}, C_2) = 2$.
Then either an induced $F_1$ arises on the vertex set $C_1\cup C_2\cup \{u^{-1},t_2'\}$ (if $c_1\sim c_2$) or an induced $H_1$ arises on the vertex set $C_1\cup C_2\cup \{u^{-1},t_1',t_2'\}$ (otherwise).
Therefore, $u^{-1}\nsim c_2$ or $u^{-1}\nsim c_2'$, as claimed.

Since $u^{-2}$ dominates $C_2$ but $u^{-1}$ does not,
we infer that $u^{-1}\neq u^{-2}$.

Next we prove that $u^{-1}\nsim u^{-2}$.
Suppose for a contradiction that $u^{-1}\sim u^{-2}$. We claim that $u^{-1}\sim c_2$ or $u^{-2}\sim c_1$. Suppose to the contrary that $u^{-1}\nsim c_2$ and $u^{-2}\nsim c_1$.
Then $c_1\nsim c_2$, since otherwise an induced $4$-cycle arises on the vertex set $\{c_1,c_2,u^{-2},u^{-1}\}$.
Furthermore, $u^{-1}\sim c_2'$ or $u^{-2}\sim c_1'$, since otherwise an induced $H_2$ arises on the vertex set
$C_1\cup C_2 \cup \{t_1',u^{-1},u^{-2},t_2'\}$.
If only one of the edges $u^{-1}c_2'$ and $u^{-2}c_1'$ is present, say
$u^{-1}c_2'$, then an induced $H_1$ arises on the vertex set $C_1\cup C_2\cup\{t_1',u^{-1},u^{-2}\}$.
If both edges $u^{-1}c_2'$ and $u^{-2}c_1'$ are present, then an induced $F_1$ arises on the vertex set $C_1\cup C_2\cup \{u^{-1},u^{-2}\}$.
Both cases lead to a contradiction, thus $u^{-1}\sim c_2$ or $u^{-2}\sim c_1$, as claimed.
We may assume without loss of generality that $u^{-1}\sim c_2$.
Now we must have $c_1\nsim c_2$ and $c_1\nsim u^{-2}$, since otherwise an induced $F_1$ or $F_2$ arises on the vertex set $C_1\cup C_2\cup \{u^{-1},u^{-2}\}$, depending on whether one or both edges are present. But now an induced $H_1$ arises on the vertex set $C_1\cup C_2\cup \{t_1',u^{-1},u^{-2}\}$, a contradiction.

To complete the proof, we consider the two cases depending on whether $c_1$ is adjacent to $c_2$ or not.
Suppose first that $c_1\sim c_2$. If $u^{-1}\sim c_2'$,
then $u^{-1}\nsim c_2$ and $G$ contains an induced $4$-cycle on the vertex set $\{u^{-1},c_2',c_2,c_1\}$. Therefore, $u^{-1}\nsim c_2'$, and similarly $u^{-2}\nsim c_1'$. If $u^{-1}\sim c_2$ and $u^{-2}\sim c_1$, then an induced $F_1$ arises on the vertex set $C_1\cup C_2\cup \{u^{-1},u^{-2}\}$. It follows that $H$ contains at most one of the edges $u^{-1}c_2$ and $u^{-2}c_1$. It follows that at least one of the sets
\hbox{$N(u^{-1})\cap V_2$}, $N(u^{-2})\cap V_1$ is empty, as claimed.

Finally, suppose that $c_1\nsim c_2$.
Suppose for a contradiction that the sets
\hbox{$N(u^{-1})\cap V_2$} and \hbox{$N(u^{-2})\cap V_1$} are both non-empty.
Since $u^{-1}\nsim c_2$ or $u^{-1}\nsim c_2'$, and, similarly,
$u^{-2}\nsim c_1$ or $u^{-2}\nsim c_1'$, we infer that
$H$ contains exactly one of the edges $u^{-1}c_2$, $u^{-1}c_2'$
and exactly one of the edges $u^{-2}c_1$, $u^{-2}c_1'$.
But now, an induced $F_1$ arises on the vertex set $C_1\cup C_2\cup \{u^{-1},u^{-2}\}$, a contradiction.
\end{proof}

\begin{claim2}\label{claim:Centers0}
$c_1\nsim c_2$.
\end{claim2}

\begin{proof}
Suppose for a contradiction that $c_1\sim c_2$ and consider $K^{-1}_2$, $u^{-1}$, $K^{-2}_1$, and $u^{-2}$. Clearly, $u^{-1}\not\in C_1\cup C_2\cup \{t_1'\}$. Moreover, since $t_2'\not\in V(K^{-1}_2)$, we have $u^{-1}\neq t_2'$ and $u^{-1}\nsim t_2'$.
Also, by symmetry, $u^{-2}\not\in C_1\cup C_2\cup \{t_2'\}$. Since $t_1'\not\in V(K^{-2}_1)$, we have $u^{-2}\neq t_1'$ and $u^{-2}\nsim t_1'$.
Furthermore, by Claim~\ref{claim:TwoPaths} we may assume without loss of generality that $N(u^{-1})\cap V_2=\emptyset$.

Let $P^{-1}=(u^{-1}=u_1,u_2,\ldots,u_k)$, with $u_k\in V_2'=C_2\cup\{t_2'\}$ be a shortest $u^{-1},V_2'$-path in $K^{-1}_2$, and similarly,
let $P^{-2}=(u^{-2}=v_1,v_2,\ldots,v_\ell)$, with $v\in V_1'=C_1\cup\{u^{-1},t_1'\}$ be a shortest $u^{-2},V_1'$-path in $V(K^{-2}_1)$.
The fact that $N(u^{-1})\cap V_2=\emptyset$ implies that $k\ge 3$ and since $u^{-2}\not\in C_2\cup \{t_2'\}$, we have $\ell \ge 2$.

Since $u^{-1}\nsim c_2$, we infer that vertex $c_1$ dominates $P^{-1}$ or otherwise $G$ would contain an induced cycle of length at least $4$.

\begin{sloppypar}
Suppose that $u_{k-1}\sim c_2'$.
To avoid an induced $4$-cycle on the vertex set $\{c_2,c_2',c_1, u_{k-2}\}$, we infer that $u_{k-2}\sim c_2$.
We must have $k=3$ since if $k \ge 4$, then the vertex set $C_2\cup \{u_{k-1}, c_1, u_{k-2}, u_{k-3}\}$ induces a copy of $F_1$.
But now, an induced copy of $F_1$ arises either on the vertex set $C_1\cup C_2\cup \{u_1, u_2\}$ (if $c_1'\nsim u_2$),
or the vertex set $C_1\cup C_2\cup \{u_2, t_1'\}$ (otherwise), a contradiction.
Therefore, $u_{k-1}\nsim c_2'$.
\end{sloppypar}

Suppose that $u_{k-1}\sim t_2'$. In this case, the vertex set $V_2'\cup \{u_{k-1}, u_{k-2}, c_1\}$ induces a copy of either $F_1$
(if $u_{k-1}\nsim c_2$) or of $F_2$ (otherwise), a contradiction.
Therefore, $u_{k-1}\nsim t_2'$. Consequently, $u_k = c_2$.

Suppose that $u^{-2}\sim c_1$.
If in addition $u^{-2}\nsim u_{k-1}$, then also $u^{-2}\nsim u_{k-2}$ (since otherwise the vertex set $\{u_{k-2}, u_{k-1}, c_2, u^{-2}\}$ would induce a $4$-cycle),
but now, the vertex set $\{u_{k-2},u_{k-1},c_1,c_2,c_2',u^{-2}\}$ induces a copy of $F_1$, a contradiction.
Therefore, $u^{-2}\sim u_{k-1}$.
Let $u_i$ be the neighbor of $u^{-2}$ on $P^{-1}$ minimizing $i$. Since $u_1\nsim u^{-2}$, we have $i>2$.
Moreover, since $u^{-2}\sim u_{k-1}$, we have $i<k$. But now, the vertex set $C_2\cup \{u_{i-1},c_1,u_i,u^{-2}\}$
induces either a copy of $F_1$ (if $u_i\nsim c_2$) or of $F_2$ (otherwise), a contradiction.
Therefore, $u^{-2}\nsim c_1$.

Since $u^{-2}\nsim c_1$, we can now apply symmetric arguments as for $P^{-1}$ to deduce that $v_\ell=c_1$ and that $c_2$ dominates $P^{-2}$.

\begin{sloppypar}
Suppose first that $V(P^{-1}) \cap V(P^{-2}) = \emptyset$.
To avoid an induced $4$-cycle on the vertex set $\{u_{k-2},c_2,c_1,v_{\ell-2}\}$, we infer that $u_{k-2}\nsim v_{\ell-2}$.
Suppose that $u_{k-1}\nsim v_{\ell-1}$. Then also $u_{k-1}\nsim v_{\ell-2}$ (since otherwise we would have an induced $4$-cycle on the vertex set
$\{u_{k-1},v_{\ell-2},v_{\ell-1},c_1\}$) and by a symmetric argument also $u_{k-2}\nsim v_{\ell-1}$.
But now, we have an induced $F_1$ on the vertex set $\{u_{k-2},c_1,u_{k-1},c_2,v_{\ell-1},v_{\ell-2}\}$.
Thus, $u_{k-1}\sim v_{\ell-1}$.
Moreover, we have either $u_{k-2}\sim v_{\ell-1}$ or $v_{\ell-2}\sim u_{k-1}$, since otherwise an induced $F_2$ arises on the vertex set $\{c_1,v_{\ell-1},v_{\ell-2},c_2,u_{k-1},u_{k-2}\}$. Without loss of generality, assume that $u_{k-2}\sim v_{\ell-1}$.
However, an induced copy of $F_1$ arises on vertex set
$\{u_{k-2}, c_1, v_{\ell-1}, v_{\ell-2}, c_2, v_{\ell-3}\}$ (where if $\ell = 3$ we define $v_0 = c_2'$).
This contradiction shows that $V(P^{-1}) \cap V(P^{-2}) \neq \emptyset$.
\end{sloppypar}

Since $v_\ell=c_1$ and due to the minimality of $P^{-2}$, we have $N(c_1)\cap V(P^{-2}) = \{v_{\ell-1}\}$.
On the other hand, since $c_1$ dominates $P^{-1}$, we have $N(c_1)\cap V(P^{-1}) = V(P^{-1})$.
Therefore $\emptyset\neq V(P^{-2})\cap V(P^{-1}) = V(P^{-2})\cap \big(N(c_1)\cap V(P^{-1})\big) = \big(N(c_1)\cap V(P^{-2})\big) \cap V(P^{-1}) = \{v_{\ell-1}\}\cap V(P^{-1})\subseteq \{v_{\ell-1}\}\,,$
which yields $V(P^{-1})\cap V(P^{-2}) = \{v_{\ell-1}\}$. A symmetric argument implies that $V(P^{-1})\cap V(P^{-2}) = \{u_{k-1}\}$; in particular,
$v_{\ell-1} = u_{k-1}$.
To avoid an induced $4$-cycle on the vertex set $\{u_{k-2},c_1,c_2, v_{\ell-2}\}$, we infer that $u_{k-2}\nsim v_{\ell-2}$.
But now, an induced copy of $F_1$ arises on vertex set $\{u_{k-3} ,u_{k-2}, c_1, u_{k-1}, c_2, v_{\ell-2}\}$
(where if $k = 3$ we define $u_0 = c_1'$).
This contradiction completes the proof of Claim~\ref{claim:Centers0}.
\end{proof}

By Claim~\ref{claim:intersectionEmpty}, we have $V_1\cap V_2= \emptyset$.
By Assumptions~\ref{ass1} and~\ref{ass2} and Claims~\ref{claim:noTedge}, \ref{claim:TipsCenters0}, and~\ref{claim:Centers0} we have $e(V_1,V_2)=0$. However, since $G$ is connected, there exists a path connecting the two diamonds $D_1$ and $D_2$. In particular, we will again consider $K^{-1}_2$, $u^{-1}$, $K^{-2}_1$, and $u^{-2}$, and analyze the possible interrelations between two particular paths to produce a forbidden induced subgraph.

Recall that by Assumption~\ref{ass3} we have $t_1'\notin V(K^{-1}_2)$ and $t_2'\notin V(K^{-2}_1)$. Furthermore, since neither of $c_2$, $c_2'$, $t_2$, $t_2'$ dominates $C_1$, we have $u^{-1}\not\in V_2$, and
Claim~\ref{claim:TwoPaths} implies that $u^{-1}\not\in V_1$. Similarly, $u^{-2}\not\in V_1\cup V_2$. Recall that Claim~\ref{claim:TwoPaths} also implies that $u^{-1}\neq u^{-2}$, $u^{-1}\nsim u^{-2}$, and that we may assume without loss of generality that $N(u^{-1})\cap V_2=\emptyset$.

Let $P^{-1}=(u^{-1}=u_1,u_2,\ldots,u_k)$, with $u_k\in C_2$, be a shortest $u^{-1},C_2$-path in $K^{-1}_2$, and let $P^{-2}=(u^{-2}=v_1,v_2,\ldots,v_\ell)$, with $v_\ell\in C_1$, be a shortest $u^{-2},C_1$-path in $K^{-2}_1$. We may assume that $u_k = c_2$ and $v_\ell = c_1$. The fact that $N(u^{-1})\cap V_2=\emptyset$ implies that $k\ge 3$ and since $u^{-2}\not\in C_2$, we have $\ell \ge 2$.

\begin{claim2}\label{claim:ell}
$\ell \ge 3$.
\end{claim2}

\begin{proof}
Suppose that $\ell = 2$. Then, $u^{-2}\sim c_1$ and hence $u^{-2}\nsim c_1'$.
To avoid a long induced cycle, we infer that
$u_{k-1}\sim u^{-2}$ and
$u_{k-1}\sim c_1$. In particular, $u_{k-1}\neq t_2'$.
Moreover, $u_{k-1}\sim t_2'$ since otherwise
the vertex set $C_2\cup \{t_2',u^{-2},u_{k-1},c_1\}$ induces a copy of either $F_1$ (if $u_{k-1}\nsim c_2'$)
or $F_2$ (otherwise).
But now $u^{-2}$ and $t_2'$ are in the same component of $G-C_2$, contradicting the fact that $u^{-2}\in V(K^{-2}_1)$ and $t_2'\notin V(K^{-2}_1)$. This contradiction implies that $\ell \ge 3$.
\end{proof}

\begin{claim2}\label{claim:uv}
$u_{k-1}\neq v_1$ and $v_{\ell-1}\neq u_1$.
\end{claim2}

\begin{proof}
Suppose for a contradiction that $u_{k-1} = v_1$. Recall that $v_1 = u^{-2}$.
By the minimality of $P^{-1}$, we have $c_2\nsim u_j$ and $c_2'\nsim u_j$ for every $j\in \{1,\ldots, k-2\}$.
Furthermore, since $u_1 = u^{-1}\nsim u^{-2} = u_{k-1}$, we have $k\ge 4$.
Since $u^{-2}$ and $t_2'$ are in different components of $G-C_2$, we infer that $t_2'\nsim u_j$ for all $j \in \{1,\ldots, k-2\}$.
If $c_1\sim u_3$, then we obtain an induced copy of $H_i$ for some $i\ge 1$
on the vertex set $\{t_2,c_2,c_2',v_1, u_{k-2}, \ldots, u_j, u_{j-1}, u_{j-2},c_1\}$, where $j\in \{3,\ldots, k\}$ is the maximum index such that
$c_1\sim u_j$. (Note that $j\le k-2$.) Therefore, $c_1\nsim u_3$, and to avoid a long induced cycle, also
$c_1\nsim u_j$ for $j\ge 4$.
A similar argument shows that $c_1'\nsim u_j$ for $j\ge 3$.
If $c_1\nsim u_2$ and $c_1'\nsim u_2$, then we obtain an induced copy of some $H_i$ on the vertex set
$V(P^{-1})\cup C_1\cup C_2\cup \{t_1',t_2'\}$.
If $c_1\sim u_2$ and $c_1'\nsim u_2$ (or vice-versa), then an induced copy of some $H_i$ arises
on the vertex set $V(P^{-1})\cup C_1\cup C_2\cup \{t_2'\}$, and if
$c_1\sim u_2$ and $c_1'\sim u_2$, then
an induced copy of some $H_i$ arises
on the vertex set $(V(P^{-1})\setminus \{u_1\})\cup C_1\cup C_2\cup \{t_1',t_2'\}$.
This contradiction shows that $u_{k-1}\neq v_1$.

Using Claim~\ref{claim:ell} we see that the paths $P^{-1}$ and $P^{-2}$ appear symmetrically. Thus, similar arguments as above imply that $v_{\ell-1}\neq u_1$.
\end{proof}

Property~$(\ref{ppt4})$ implies the following.

\begin{claim2}\label{claim:non-disjoint}
$V(P^{-1}) \cap V(P^{-2}) \neq \emptyset$.
\end{claim2}

We are now ready to complete the proof of the Diamond Lemma.
Let $r\in \{1,\ldots, k\}$ be the minimum index such that $u_r\in V(P^{-2})$. Note that $r<k$, since $u_k\in C_2$ and $C_2\cap V(P^{-2}) = \emptyset$.
Let $s\in \{1,\ldots, \ell\}$ be the index such that $u_r = v_s$.
If $r = 1$, then $u_1 = v_{\ell-1}$, contradicting Claim~\ref{claim:uv}.
Therefore, $r\ge 2$. Similarly, if $s = 1$, then $v_1 = u_{k-1}$, again contradicting Claim~\ref{claim:uv}. Therefore, $s\ge 2$.

Consider the path $Q = (u_1,\ldots, u_r = v_s, v_{s-1}, \ldots, v_1)$.
Let $D$ and $D'$ be the subgraphs of $G$ induced by $\{t_1',c_1,c_1',u_1\}$ and $\{t_2',c_2,c_2',v_1\}$, respectively. Notice that $D$ and $D'$ are diamonds. We will refer to tips $u_1$ and $v_1$ as the \emph{roots} of $D$ and $D'$, respectively. Then, $Q$ is a path connecting the two roots.
Moreover, by Assumption~\ref{ass3} we have $t_1'\notin V(K^{-1}_2)$
and $V(Q)\subseteq V(K^{-1}_2)$, we infer that $t_1'$ has no neighbors on $Q$. Similarly, $t_2'$ has no neighbors on $Q$.

We may also assume that $Q$ is an induced path; otherwise, we replace $Q$ with a shortest $t_1',t_2'$-path in $G[V(Q)]$.
To complete the proof, we will show that $G$ is not $\{F_1,F_2,H_1,H_2,\dots\}$-free.
We say that an induced subgraph $H$ of $G$ is a \emph{weakly induced} $H_n$ if $H$ has a spanning subgraph $H_n$ with $n\ge 1$ consisting of two diamonds and a path connecting them such that, assuming notation from Fig.~\ref{fig:H}, the following holds:
\begin{enumerate}[(i)]
\item each of the two diamonds is induced in $G$,
\item there are no edges in $G$ connecting a vertex from one diamond with a vertex from another diamond, except perhaps edges incident with their roots (if $n = 1$) or the unique edge on the path connecting the two roots (if $n = 2$),
\item the path connecting the two diamonds is induced in $G$, and
\item vertices $x_1$ and $z_1$ do not have any neighbors on the path.
\end{enumerate}

\begin{figure}[!ht]
  \begin{center}
\includegraphics[width=0.55\linewidth]{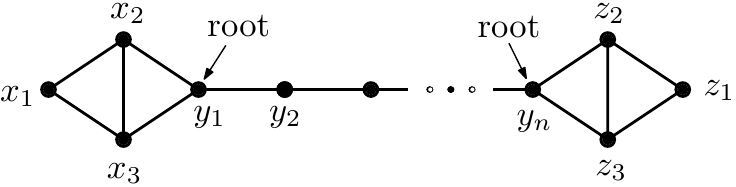}
  \end{center}
\caption{A weakly induced $H_n$}
\label{fig:H}
\end{figure}

Note, in particular, that every weakly induced $H_n$ for $n\in \{1,2\}$ is isomorphic to $H_n$.

The above considerations show that the subgraph of $G$ induced by $V(D)\cup V(D')\cup V(Q)$ contains a weakly induced $H_n$. Choose one such induced subgraph, say $H$, with minimum value of $n$, and let $F$ be the corresponding spanning subgraph of $H$ isomorphic to $H_n$.
To complete the proof, we will now show that either $H$ equals $F$ or $G$ contains an induced $F_1$ or $F_2$. Suppose that this is not the case. The only possible edges that can be present in $H$ but not in $F$ are those connecting one of the vertices $x_2, x_3, z_2, z_3$ with one of the vertices in the set $\{y_2,\ldots, y_{n-1}\}$.

Let us first show that for each $i\in \{2,\ldots, n-1\}$,
at most one of $x_2$ and $x_3$ is adjacent to $y_i$.
Suppose that $x_2\sim y_i$ and $x_3\sim y_i$ for some $i\in \{2,\ldots, n-1\}$. Then $y_i\sim z_2$ or $y_i\sim z_3$, since otherwise the subgraph of $G$ induced by $\{x_1,x_2,x_3,y_i,\ldots, y_n,z_1,z_2,z_3\}$ would be a weakly induced $H_{n-i+1}$, contradicting the minimality of $H$.
If $y_i\sim z_2$ and $y_i\sim z_3$, then the vertex set
$\{x_1,x_2,x_3,y_i,z_1,z_2,z_3\}$ induces an $H_1$ in $G$.
We may thus assume that $y_i$ is adjacent only to one of $z_2,z_3$, say to $z_3$. If $i = n-1$, then the vertex set $\{x_1,x_2,x_3,y_{n-1},y_n,z_2,z_3\}$ induces an $H_1$ in $G$.
If $i\le n-2$, then the fact that $G$ is chordal implies that
$z_3\sim y_{j}$ for all $j\in \{i,\ldots, n\}$, and the vertex set
$\{x_1,x_2,x_3,y_i,y_{i+1},y_{i+2},z_2\}$ induces an $H_1$ in $G$.
This contradiction shows that for each $i\in \{2,\ldots, n-1\}$,
at most one of $x_2$ and $x_3$ is adjacent to $y_i$.

Next, we argue that at least one of $x_2$ and $x_3$ is not adjacent to any
vertex $y_i$ with $i\in \{2,\ldots, n-1\}$.
Indeed, if $x_2\sim y_r$ and $x_3\sim y_s$, with $2\le r\le s \le n-1$ (say),
then $r<s$ and the fact that $G$ is chordal implies that
$x_3\sim y_{j}$ for all $j\in \{2,\ldots, s\}$, contradicting the fact that
at most one of $x_2$ and $x_3$ is adjacent to $y_r$. Therefore, we may assume without loss of generality that $x_2$ has no neighbors in the set $\{y_2,\dots,y_{n-1}\}$. Similarly, we may assume that $z_2$ has no neighbors in the set $\{y_2,\dots,y_{n-1}\}$.

Let $r\in \{1,\ldots, n-1\}$ be the maximum index such that $x_3\sim y_r$.
Similarly, let $s\in \{2,\ldots, n\}$ be the minimum index such that $z_3\sim  y_s$. If $r = 1$ and $s = n$, then $H = F$ and we are done.
Thus, we may assume without loss of generality that $r\ge 2$.
In particular, this implies that $x_3\sim y_2$.
If $y_2\nsim z_3$, then the subgraph of $G$ induced by $\{x_2,x_3,y_1,\ldots, y_n,z_1,z_2,z_3\}$ is a weakly induced $H_{n-1}$, contradicting the minimality of $H$. Therefore, $y_2\sim z_3$, or, equivalently, $s = 2$.
A similar argument shows that $r = n-1$.
Now, if $n = 3$, then the vertex set
$\{x_2,x_3,y_1,y_2,y_3,z_2,z_3\}$ induces an $H_1$ in $G$, and if $n\ge 4$,
then the vertex set
$\{x_2,x_3,y_1,y_2,y_3,z_3\}$ induces an $F_1$ in $G$.
This contradiction completes the proof of the Diamond Lemma.
\qed
\end{cproof}

\appendix

\begin{sloppypar}
\section*{Appendix: A non-connected-domishold split graph
whose cutset \hbox{hypergraph} is $2$-asummable}
\end{sloppypar}

Based on an example due to Gabelman~\cite{Gabelman}, Crama and Hammer proposed in the proof of~\cite[Theorem 9.15]{MR2742439} an example of a
$9$-variable $2$-asummable positive Boolean function $f$ that is not threshold. From this function we can derive a split graph $G=(V,E)$ on $71$ vertices, as follows. Let $V = K\cup I$ where $K=\{v_1,\dots,v_9\}$ is a clique and $I=V(G)-K$ is an independent set. To define the edges between $K$ and $I$, we first associate a non-negative integer weight to each vertex, as follows:
$w(v_1)=14$, $w(v_2)=18$, $w(v_3)=24$, $w(v_4)=26$, $w(v_5)=27$, $w(v_6)=30$, $w(v_7)=31$, $w(v_8)=36$, $w(v_9)=37$, and $w(v)=0$ for all $v\in I$. Let $\mathcal{S}$ be the set of all subsets $S$ of $K$ such that $w(S)\ge 82$ and let $S_1=\{v_1,v_6,v_9\}$, $S_2=\{v_2,v_5,v_8\}$, and $S_3=\{v_3,v_4,v_7\}$. (Note that $w(S_i)=81$ for all $i\in [3]$.)
Let $\mathcal{H}$ be the hypergraph with vertex set $K$ and hyperedge set
given by the inclusion-wise minimal sets in $\mathcal{S}\cup \{S_1,S_2,S_3\}$.
It can be verified that $\mathcal{H}$ has precisely $62$ hyperedges (including $S_1$, $S_2$, and $S_3$).\footnote{The following is the list of sets (omitting commas and brackets) of indices of the elements of the $62$ inclusion-wise minimal hyperedges of $\mathcal{H}$: 169, 179, 189, 258, 259, 268, 269, 278, 279, 289, 347, 348, 349, 357, 358, 359, 367, 368, 369, 378, 379, 389, 456, 457, 458, 459, 467, 468, 469, 478, 479, 489, 567, 568, 569, 578, 579, 589, 678, 679, 689, 789, 1234, 1235, 1236, 1237, 1238, 1239, 1245, 1246, 1247, 1248, 1249, 1256, 1257, 1267, 1345, 1346, 1356, 2345, 2346, 2356.}
The edges of $G$ between vertices of $I$ and $K$ are defined so that set of the neighborhoods
of the $62$ vertices of $I$ is exactly the set of hyperedges of $\mathcal{H}$.

To show that $G$ is not CD, it suffices, by Proposition~\ref{prop:c-dom-graphs}, to show that the cutset hypergraph
is not threshold. In the proof of Theorem 9.15 in~\cite{MR2742439} it is shown that the function $f$ is not threshold, by showing that $f$ is $3$-summable. This corresponds to the fact that the cutset hypergraph of $G$ is $3$-summable, as can be observed by noticing that condition~\eqref{eq:k-summable hypergraph} is satisfied for $k = r = 3$ and for
the sets $A_i = S_i$ for all $i\in [3]$ and $B_1 = \{v_1,v_7,v_8\}$, $B_2 = \{v_2,v_4,v_9\}$, and $B_3 = \{v_3,v_5,v_6\}$.
On the other hand, the fact that $f$ is $2$-asummable implies that the cutset hypergraph of $G$ is $2$-asummable.

\medskip

\section*{Acknowledgments}

\begin{sloppypar}
The authors are grateful to an anonymous reviewer for insightful suggestions, to Endre Boros, Yves Crama, Vladimir Gurvich, Pinar Heggernes, Haiko M\"uller, and Vito Vitrih for helpful discussions, and to
Dieter Kratsch and Haiko M\"uller for providing them with a copy of paper~\cite{MR1284728}. The work for this paper was partly done in the framework of a bilateral project between Argentina and Slovenia, financed by the Slovenian Research Agency (BI-AR/$15$--$17$--$009$) and MINCYT-MHEST (SLO/14/09). The authors acknowledge financial support from the Slovenian Research Agency (research core funding No.~I$0$-$0035$ and P$1$-$0285$, and projects N$1$-$0032$, J$1$-$6720$, and J$1$-$7051$).
\end{sloppypar}

\def\ocirc#1{\ifmmode\setbox0=\hbox{$#1$}\dimen0=\ht0 \advance\dimen0
  by1pt\rlap{\hbox to\wd0{\hss\raise\dimen0
  \hbox{\hskip.2em$\scriptscriptstyle\circ$}\hss}}#1\else {\accent"17 #1}\fi}

\end{document}